\documentclass[a4paper,12pt]{article}
\usepackage{parskip}
\usepackage[normalem]{ulem}
\usepackage{amsfonts,amsmath,amssymb,xcolor,hyperref,amsthm,graphicx,comment,enumerate}
\title{Nekhoroshev theory and discrete averaging}

\author{V.~Gelfreich$^1$ and A.~Vieiro$^{2,3}$\\[4pt]
$^1$ \small Mathematics Institute, University of Warwick, \small Coventry CV4 7AL, UK\\
\small{\tt v.gelfreich@warwick.ac.uk}\\[4pt]
$^2$ \small Departament de Matem\`atiques i Inform\`atica,\\ \small Universitat de Barcelona, Gran Via 585, 08007 Barcelona, Spain\\
$^3$ \small Centre de Recerca Matem\`atica (CRM),\\ \small Campus Bellaterra, 08193 Bellaterra, Spain\\
\small{\tt vieiro@maia.ub.es}
}

\providecommand{\keywords}[1]
{
  \small	
  \textbf{\textit{Keywords---}} #1
}

\newcommand{\val}{\mathrm{val}}

\newcommand{\N}{\mathbb{N}}
\newcommand{\R}{\mathbb{R}}
\newcommand{\C}{\mathbb{C}}
\newcommand{\T}{\mathbb{T}}
\newcommand{\Z}{\mathbb{Z}}
\newcommand{\Q}{\mathbb{Q}}
\newcommand{\J}{\mathrm J}
\newcommand{\I}{\mathrm I}
\newcommand{\Tf}{\mathrm T}

\newtheorem{thm}{Theorem}[section]
\newtheorem{lemma}[thm]{Lemma}
\newtheorem{remark}[thm]{Remark}

\begin{document}

\maketitle	
\begin{abstract}
This paper contains a proof of the Nekhoroshev theorem for quasi-integrable
symplectic maps. 
In contrast to the classical methods, our proof 
is based on the discrete averaging method and does not rely on transformations 
to normal forms. 
At the
centre of our arguments lies the theorem on 
embedding of a near-the-identity symplectic map into
an autonomous Hamiltonian flow with exponentially small error.
\end{abstract}
\hspace{1pt}


\section{Introduction}

If a quasi-integrable Hamiltonian system satisfies 
suitable non-degeneracy assumptions,  the
KAM theorem states that invariant tori occupy the major part of the phase space and the Lebesgue measure of the complement to the set of the tori converges to zero when a perturbation parameter vanishes \cite{Arn63,BiasChi18,Nei81,Pos82}. This complement is dense in the phase space.
Moreover, if the number of degrees of
freedom is three or larger,  the complement  is a connected set.
Therefore, in contrast to the case of two degrees of freedom, the KAM
theory does not prevent  existence of trajectories which may travel large distances inside an energy level. This phenomenon is known under the name of Arnold diffusion \cite{Arnold64}. The maximal speed of Arnold diffusion is bounded
by Nekhoroshev estimates \cite{Nekhoroshev77}, 
which state that action variables oscillate near their initial values for exponentially long times.

Hamiltonian perturbation theory studies both  systems of Hamiltonian
equations and symplectic maps.
A $2d$-dimensional symplectic quasi-integrable map can be
seen as an isoenergetic Poincar\'e section of a Hamiltonian
flow with $d+1$ degrees of freedom \cite{KuksinP94},
and a Nekhoroshev theorem for maps is usually derived from a Nekhoroshev theorem for flows, for 
example, from the results of \cite{LochakN92},
although  direct proofs for maps are also available (see e.g.  \cite{Guzzo04}).

In the literature one can find two different strategies for proving
a Nekhoroshev type estimate. The first strategy,  originally proposed by Nekhoroshev 
\cite{Nekhoroshev77,Nekhoroshev79}, relies on a careful study of normal forms
for resonances of all possible multiplicities. An alternative strategy was proposed by Lochak   \cite{Lochak92} who showed that the analysis can be restricted to
resonances of the highest multiplicity only.
In the convex case
both strategies lead to optimal stability coefficients \cite{LochakN92,Poschel93}. 

\goodbreak

In this paper we provide a new proof of the Nekhoroshev theorem for
quasi-integrable symplectic analytic maps  under the convexity assumption,
without reducing the problem to a Hamiltonian flow.
A part of our proof follows the Lochak-Neishtadt approach \cite{LochakN92} but, in
contrast to the traditional approach,  our method does not rely on
transformations to normal forms. Instead we propose a direct reconstruction of
slow observables from  iterates  of the
quasi-integrable map in original coordinates with the help of a discrete averaging method.
Our construction can be applied to an individual map and  provides
explicit values for constants in the estimates. 
Therefore, our method can be used
for numerical analysis of Arnold diffusion and 
for developing new tools for visualisation of the dynamics \cite{GelfreichV18}.

Our proof of the Nekhoroshev theorem includes a refined version of  classical
Neishtadt's theorem~\cite{Neishtadt84} which may be of independent interest.
Neishtadt proved that if a member of a smooth 
near-the-identity  family of analytic symplectic maps is sufficiently close to the identity, then it can be
approximated by an autonomous Hamiltonian flow with an exponentially small
error.  Neishtadt's proof relies on representing the maps as time-one maps of
time-periodic flows and classical averaging. Alternatively the theorem can be proved
using Moser's analysis of the formal interpolating flow
\cite{BenettinG85,Moser68}. By contrast, our construction is applicable to
individual maps. We will  show that the approximation error can be controlled by
the ratio of two natural parameters: one characterises the size of a complex
neighbourhood where the map is close to the identity and the second one is the
distance from the map to the identity in a suitably chosen norm.
Our construction  is based
on the discrete averaging and has an additional important advantage: the
construction is explicit in terms of iterates of the map and can be easily
implemented numerically~\cite{GelfreichV18}.

The paper is organized in the following way. Section~\ref{Se:NekhoroshevThm}
presents
the statement of the Nekhoroshev theorem and explains its derivation
from the analysis of stability near fully resonant tori.
Section~\ref{Se:strategy} contains a proof of some elementary bounds
and explains the strategy of our proof.
 Section~\ref{Se:interpolatingVF} contains the necessary details of 
 the theory of  interpolating vector fields. 
 In Section~\ref{Se:NeishtadtThm} we prove
our refined version of Neishtadt's theorem, the  embedding theorem of a near-the-identity map into an autonomous Hamiltonian flow
up to an exponentially small error. The exponential bounds for stability time
are proved  in Section~\ref{Sect:Nekhoroshev}. 
In this way
Sections~\ref{Se:interpolatingVF}, \ref{Se:NeishtadtThm} and
\ref{Sect:Nekhoroshev} represent a self-contained proof of the Nekhoroshev theorem
stated in Section~\ref{Se:NekhoroshevThm}. 
Finally, in Section~\ref{Se:Nucleus} we
consider a small region around a resonance, that we called a nucleus of 
resonance, and discuss some improved bounds for stability times in these
regions of the phase space. In particular we discuss how the stability times scale
when $\varepsilon$ approaches zero.
Our final comments and conclusions are given in Section~\ref{Se:conclusion}.

\section{Nekhoroshev theorem for a near-integrable map
\label{Se:NekhoroshevThm}}

The main object of this paper is a one-parameter family
$F_\varepsilon:(I,\varphi)\mapsto (\bar I,\bar \varphi)$ of 
real-analytic exact symplectic maps 
of the form 
\begin{equation}\label{Eq:feps}
\left\{
\begin{aligned}
	\bar I&=I+\varepsilon a(I,\varphi),\\
	\bar \varphi&=\varphi+\omega(I)+\varepsilon b(I,\varphi) \pmod1,
\end{aligned}	
\right.
\end{equation}
where the functions $a,b$ are periodic  in $\varphi\in\T^d=\R^d/\Z^d$,
$I$ is a vector in $\R^d$ 
and $\varepsilon\ge 0$ is a perturbative parameter.
At $\varepsilon=0$ the map is integrable and $F_0:(I,\varphi)\mapsto (\bar I,\bar \varphi)$ is simply given by
\begin{equation}\label{Eq:f0}
\left\{
\begin{aligned}
	\bar I&=I,\\
	\bar \varphi&=\varphi+\omega(I) \pmod1.
\end{aligned}	
\right.
\end{equation}
Obviously the variable $I$ remains constant along trajectories of $F_0$ and
the equation $I=I_0$ defines an invariant torus with the frequency vector
$\omega(I_0)$.
Trajectories of $F_\varepsilon$  with very close initial conditions are capable of
separating from each other exponentially fast with  stability times 
 $T_L \sim \varepsilon^{-1/2}$, 
 a natural Lyapunov time for the system.  
 The Nekhoroshev estimates address substantially longer timescales.
Given an initial condition $(I_0,\varphi_0)$  let $(I_k,\varphi_k) = F_\varepsilon^k(I_0,\varphi_0)$, $k\in\mathbb Z$, be the corresponding trajectory.
A Nekhoroshev estimate states that, for $|\varepsilon| \leq \varepsilon_0$,
$$
|I_k - I_0| \leq R(\varepsilon) \quad\text{ for  $|k| \leq T_\varepsilon$} ,
$$
where the radius of confinement $R(\varepsilon)=\mathcal{O}(\varepsilon^\beta)$ and the
stability time $T_\varepsilon\sim \exp(c/\varepsilon^\alpha)$, 
for suitable stability exponents $0 < \alpha,\beta
\leq 1$. 

It is well known that the long time behaviour of trajectories
is sensitive to the smoothness class of the map $F_\varepsilon$ \cite{BamLan21,Boune10,BounemouraN12,MarcoSauz02}. In this paper
we assume that $F_\varepsilon$ is real-analytic although some of
our arguments are applicable to other smoothness classes.
Without loosing in generality we assume that $F_\varepsilon$
is real-analytic in $B_R \times \T^d$, where
$B_R\subset \R^d$ is a ball of radius $R>0$
centered at  some  point of $\R^d$.
We assume that for all $\varepsilon$  
with $0\le \varepsilon\le \varepsilon_0$, 
the lift of the map 
have an analytic continuation onto the complex 
neighbourhood $\mathcal{D}_{F}$ of $B_R \times \R^d$ given by
\begin{equation}\label{Eq:D_F}
    \mathcal{D}_{F}=\left\{\;(I,\varphi) \in \mathbb{C}^{2d}\;:\;\text{dist}(I,B_R) \le \sigma,\; |\operatorname{Im}(\varphi)|\le r\;\right\}
\end{equation}
for some $\sigma,r>0$ independent of $\varepsilon$. 
We hide the dependence of $a$ and $b$ on $\varepsilon$
to simplify notations. Our proof is not sensitive to the nature of this dependence and it is sufficient to assume that supremum norms of $a$ and $b$ are bounded uniformly in $\varepsilon$.

It is also well known that the long time stability of action variables $I$ depends 
on the properties of the unperturbed frequency map $\omega$ \cite{BenettinG85,GuzChiBen16, Nekhoroshev77,Nekhoroshev79}. Since the map is symplectic, $\omega$ is a gradient of a scalar function.
We assume that $\omega=h_0'$ where the function
$h_0$ is strongly convex on the intersection of $\mathcal D_F$ with the real subspace, i.e. there is $\nu >0$ such that for every real $I$ in the domain of $h_0$
one has 
\begin{equation} \label{convexity}
h_0''(I)(v,v) = 
(\omega'(I)v ) \cdot v 
\geq \nu \, 
v \cdot v 
\end{equation}
for all $v\in\R^d$.
The convexity assumption means that $\nu$ is a lower bound for the spectrum of the Hessian matrices $h_0''(I)$ for real values of $I$.
Note that a function $h_0$ is strongly convex on a convex set
 iff for any $I,J$
\begin{equation}\label{Eq:strongconv}
(h_0'(I)-h_0'(J)) \cdot (I-J)
\ge \nu \, 
(I-J)\cdot (I-J)
\end{equation}
with the same convexity constant $\nu$.

Under these assumptions we will prove the Nekhoroshev
estimates with optimal exponents $\alpha=\beta=\frac1{2(d+1)}$.

\begin{thm}[Nekhoroshev theorem]   \label{Thm:Nekhoroshevestimates}
If 
a real-analytic exact symplectic map $F_\varepsilon$ satisfies the  assumptions stated above  and
$h_0$ is strongly convex on a
real neighbourhood of $B_R$, then
there are positive constants 
$c_1,c_2,c_3$
such that for every initial condition
 $(I_0,\varphi_0)\in B_R\times \mathbb T^d$
 one has
\[
|I_k-I_0|< c_1 \varepsilon^{1/2(d+1)}
\quad \text{ for } \quad
0\le k\le T_\varepsilon=c_2 \exp\left(c_3\varepsilon^{-1/2(d+1)}\right).
\]
\end{thm}

We derive the Nekhoroshev theorem from a statement which provides more detailed
information about the stability of actions.

\begin{thm}[long term stability of actions]\label{Thm:longstab}
Under the assumptions of the Nekhoroshev
theorem, 
there are constants $\gamma_0>0$ and $r_0\in(0,1)$
with the following property. 
For every $\gamma\ge \gamma_0$ there are positive constants $\varepsilon_0$,
 $c_2,c_3$ such that if 
$0<\varepsilon \leq \varepsilon_0$, 
$n< \varepsilon^{-d/2(d+1)}$,
and  $I_*\in B_R$ corresponds to a fully resonant
unperturbed torus with $n\omega (I_*)\in\Z^d $,  
then any trajectory with initial conditions satisfying 
$|I_0-I_*|<r_0\rho_n$ and $\varphi_0 \in \T^d$
satisfies the inequality
 \[
|I_k-I_*|< \rho_n 
\qquad \text{for} \quad 
0\le k \le T_\varepsilon=c_2 \exp\left(c_3\varepsilon^{-1/2(d+1)}\right),
\]
where 
\( \rho_n=\gamma n^{-1} \varepsilon^{1/2(d+1)} \).
\end{thm}

Note that we will provide explicit expressions for $\gamma_0$ and $r_0$.

In order to derive Theorem~\ref{Thm:Nekhoroshevestimates}
from Theorem~\ref{Thm:longstab} we use the ideas of
Lochak covering to show that the balls
$|I_0-I_*|<r_0\rho_n$ cover all initial conditions.
Our arguments use the Dirichlet Theorem on simultaneous
approximations in a way similar to the papers \cite{Lochak92,LochakN92,LochNN94}.

\begin{thm}[Dirichlet \cite{Cassels57}] \label{Thm:Dirichlet}
For any $\omega\in\R^d$ and any $N>1$ there are $\omega_*\in\Q^d$ and
$n\in\N$ such that $ n < N$, $n\omega_*\in\Z^d$  and
\(
|\omega-\omega_*|<\frac1{nN^{1/d}}.
\)
\end{thm}	

If the  frequency map $\omega:I\mapsto h_0'(I)$ is 
defined by a strongly convex function $h_0$,  we can prove a similar
result in the space of actions.

\begin{lemma}\label{Le:covering}
Let a convex set $U_\delta\subset \R^d$ be a $\delta$-neighbourhood of a  set $U \subset \R^d$.
If $h_0$ is strongly convex in $U_\delta$ with parameter $\nu$,
then there is $N_0=N_0(\nu,\delta)$ such that for any $N>N_0$ and any $I_0 \in
U$ there is a point $I_*\in U_\delta$ and $n\in\N$ such that $n < N$, 
$n\omega(I_*)\in\mathbb Z^d$ and 
\[|I_0 -I_*|
<\frac{\sqrt d}{
\nu n N^{1/d}}.\] 
\end{lemma}

\begin{proof}
Using the strong convexity of the function $h_0$  in the form \eqref{Eq:strongconv} we get 
\[
|\omega(I_1)-\omega(I_2)|_2\,|I_1-I_2|_2\ge
(\omega(I_1)-\omega(I_2)) \cdot (I_1-I_2) 
\ge \nu |I_1-I_2|_2^2
\]
for all $I_1,I_2\in U_\delta$ and consequently
\(|\omega(I_1)-\omega(I_2)|_2\ge \nu |I_1-I_2|_2.
\)
For the sake of convenience we  use the Euclidean norm in this bound. 
It follows that $I_1\ne I_2$ implies $\omega(I_1)\ne\omega(I_2)$ and consequently  the map $\omega:U_\delta\to \omega(U_\delta)$ is bijective.

 Now let $I_0\in U$ and $N\in\N$. The Dirichlet theorem implies that there is
$\omega_*\in \Q^d$ 
such that $n \omega_*\in\Z^d$ for some $n<N$ and $|\omega(I_0)-\omega_*|<n^{-1}N^{-1/d}$.
We note that if $|I_0-I_*|<\delta$ then
\[|\omega(I_0)-\omega(I_*)|\ge \frac{1}{\sqrt d}
|\omega(I_0)-\omega(I_*)|_2
\ge \frac{\nu}{\sqrt d} |I_0-I_*|_2
\ge
 \frac{\nu}{\sqrt d} |I_0-I_*|.\]
 Consequently, if $n^{-1}N^{-1/d}<\nu d^{-1/2} \delta$ then $\omega_*=\omega(I_*)$
 for some $I_*$ with $|I_0-I_*|<\delta$.
 Let  $N_0^{-1/d}=\nu d^{-1/2} \delta$.
 Then for any $N>N_0$ there is $I_*\in U_\delta$ such that $n\omega(I_*)\in\Z^d$
 for $n<N$ and
 \[
 |I_0-I_*|<\frac{\sqrt d}{\nu}|\omega(I_0)-\omega_*|<\frac{\sqrt d}{\nu n N^{1/d}}.\qedhere
 \]
\end{proof}

We conclude that if $\varepsilon_0^{d/2(d+1)}<1/N_0$ and $\gamma_0 r_0 > \sqrt d \nu^{-1} $,
then the balls $B(I_*,r_0\rho_n)$ with $n$ and $I_*$ such that $n<\varepsilon^{-d/2(d+1)}$ and $n\omega(I_*)\in\Z^d$ cover $B_R$.
Consequently every initial condition belongs to a neighbourhood of a resonant torus $I=I_*$ where the stability bounds of Theorem~\ref{Thm:longstab} are applicable and Theorem~\ref{Thm:Nekhoroshevestimates} follows immediately.
We note that there is no claim of uniqueness for $I_*$ and some initial conditions
may belong to several zones of stability.

\section{A priori bounds and strategy of the proof\label{Se:strategy}}

The $n^{\mathrm{th}}$ iterate of the map $F_\varepsilon$ can be written
explicitly in the form
\begin{equation} \label{Eq:fepsn}
\left\{
\begin{aligned}
	I_n&=I_0+\varepsilon \sum_{k=0}^{n-1}a(I_k,\varphi_k),\\
	\varphi_n&=\varphi_0+\sum_{k=0}^{n-1}\omega(I_k)+\varepsilon\sum_{k=0}^{n-1} b(I_k,\varphi_k) ,
\end{aligned}	
\right.
\end{equation}
where $(I_k,\varphi_k)=F_\varepsilon(I_{k-1},\varphi_{k-1})$
denote points on the trajectory with initial conditions $(I_0,\varphi_0)$.
We can slightly overload our notation by 
assuming that the angle component 
in this formula is computed without 
taking the angle modulo one.
We hope that this will not create too much confusion
as the functions $a$ and $b$ are periodic in $\varphi$.
The following simple lemma implies that 
 trajectories of $F_\varepsilon$ follow
rather closely trajectories of the unperturbed integrable map $F_0$
for times much shorter than $T_L\sim\varepsilon^{-1/2}$.

\begin{lemma}[a priori bounds]
\label{Le:apriori}
Suppose that the map $F_\varepsilon$
has an analytic continuation onto the complex domain $\mathcal D_F$
and $n\in\N$. 
If
 $(I_k,\varphi_k)
 \in \mathcal{D}_F$ for $0\le k <  n$, then
\begin{equation}
 \label{apriori}
 |I_n-I_0|\le C_1 n\varepsilon,
\qquad
|\varphi_n-\varphi_0-n\omega(I_0)|\le C_2 n^2\varepsilon.
\end{equation}
where $C_1=\|a\|$ and $C_2=\tfrac12\|\omega'\| \|a\| +\|b\|$.
\end{lemma}

\begin{proof}
Since the iterates of the initial point belong
to $\mathcal D_F$ the triangle inequality implies that
$|I_n - I_0|  \le n \varepsilon \| {a}\|$. 
Then
\begin{align*}
|\varphi_n - \varphi_0 - n \omega(I_0) | &  
\le  \sum_{k=0}^{n-1} |\omega(I_k) - \omega(I_0)| + n \varepsilon \|{b}\| 
\\ & \le \|w'\| \sum_{k=0}^{n-1} |I_k -I_0|   + n  \varepsilon\|b\| 
\\ & \le  \|w'\|  \sum_{k=0}^{n-1} k\varepsilon \|{a}\|   + n \varepsilon \|b\| 
 \le \frac{n^2\varepsilon}{2}\|\omega'\| \,\|a\| +n\varepsilon\|b\|
\end{align*}
and the desired estimate follows  immediately as $n^2\ge n$.
\end{proof}

In a way similar to Lochak-Neishtadt's proof of the Nekhoroshev theorem \cite{LochakN92},
we analyse dynamics in carefully chosen neighbourhoods of
unperturbed tori bearing periodic motions. Let $n\,\omega(I_*)\in\Z^d$ for some
$n\in\N$ and $I_*\in\R^d$. The equation $I=I_*$ defines a torus filled with
periodic orbits of the integrable map $F_0$. 
The point  $I_*$ corresponds to a resonance of the maximal
multiplicity because the set $\left\{r\in \Z^d: r \cdot \omega(I_*) = 0 \,
(\text{mod }1)\right\}$  contains the $d$-dimensional sublattice $n\Z^d$.

Since  $n\omega(I_*)\in \Z^d$ we
can consider another  lift of $F_\varepsilon^n$ defined by the
equation   
\begin{equation} \label{mapFn}
    f_\varepsilon^n:(I_0,\varphi_0)\mapsto(I_n,\varphi_n-n\omega(I_*)).
\end{equation}
Of course, the maps $F_\varepsilon^n$ and $f_\varepsilon^n$ define the same trajectories when the angle
variables are considered modulo one.
In order to prove Theorem~\ref{Thm:longstab} we  restrict our attention
to $n<N_\varepsilon$ with
\begin{equation}  \label{Nmax_nearid}
 N_\varepsilon=\varepsilon^{-d/2(d+1)},
\end{equation}
and study the dynamics of
$f_\varepsilon^n$ on the domain 
\begin{equation}\label{Eq:D(I*)}
\mathcal D_0(I_*)=B(I_*,\rho_n) \times \R^d    
\end{equation}
where the radius of the ball 
$
\rho_n=\rho_\varepsilon/n
$
with
\begin{equation}\label{Eq:rho_eps}
\rho_\varepsilon=\gamma N_\varepsilon^{-1/d}=\gamma \varepsilon^{1/2(d+1)} 
\end{equation}
and $\gamma$ is a constant independent of $n$ and $\varepsilon$.

For $\varepsilon=0$ the map  takes the form
\[
f_0^n:(I,\varphi)\mapsto (I,\varphi +n\omega(I)-n\omega(I_*)).
\]
Consequently the set defined by $I=I_*$ consists of fixed points. 
It is also easy to see that $f_0^n$ coincides
with the time-one map of the integrable flow defined by
the Hamiltonian function
\[
h_n(I)=n\bigl(h_0(I)-h_0(I_*)-\omega(I_*)\cdot(I-I_*)\bigr).
\]
Using Lemma~\ref{Le:apriori}  we will check that in $\mathcal D(I_*)$,  a
suitable complex neighbourhood of $\mathcal D_0(I_*)$, the lift $f_{\varepsilon}^n$ 
is close to the identity.

In Section~\ref{Se:NeishtadtThm} we prove a refined version of Neishtadt's theorem 
which establishes explicit bounds for the error of approximation of 
a near-the-identity symplectic map by an autonomous Hamiltonian flow.
In Section~\ref{Se:intrepol_at_Torus} we will check that this theorem
can be used to show that
$f_{\varepsilon}^n$ is exponentially close
to the time-one map of a Hamiltonian flow $\hat{X}_m = \J \nabla H_m$ where $\J$
is the standard symplectic matrix. 
Here the subscript $m$ refers to the fact
that $H_m$ is obtained from an interpolating vector field based on $m$ consecutive iterates of the map $f_\varepsilon^n$.
We will use $m \sim 1/\epsilon_n$ where
$\epsilon_n$ is the distance from $f_{\varepsilon}^n$
to the identity map in $\mathcal
D(I_*)$. Then in Section~\ref{Se:stability}  we will show that 
\[
H_m(I,\varphi) = h_n(I)+ w_{n,m}(I,\varphi),
\]
where 
the perturbative term $w_{n,m}$ is negligible
for the purpose of the stability analysis.
The convexity of $h_0$ and the Taylor formula imply that
\[
\tfrac1{2} \nu n \,|I-I_*|^2\le h_n(I)\le \tfrac d2\|h''\|n \, |I-I_*|^2.
\]
These inequalities imply that there is $r_0\in (0,1)$ such that if $I\in B(I_*,r_0\rho_n)$ then the corresponding energy level set of $H_m$ is located inside $B(I_*,\rho_n)$. Consequently, the  trajectories of the Hamiltonian flow which start in the smaller ball  will never leave the larger one.
Then we can easily conclude that a trajectory
of the map is trapped in such a neighbourhood for
exponentially long times as 
one iterate of $ f^n_\varepsilon$
changes the energy by an exponentially small quantity. So we will need exponentially many iterates to reach a change in the energy
needed to leave the domain. While the action coordinates remain
in $B(I_*,\rho_n)$ the  oscillations
of  $I$ will not exceed the diameter $2\rho_n=2\gamma\varepsilon^{1/2(d+1)}/n$.
This completes the sketch of proof of Theorem~\eqref{Thm:longstab}.

\begin{remark}
    The choices of $N_\varepsilon$ and $\rho_\varepsilon$ originate from the following reasoning. The distance to the identity in the angle component of $f_\epsilon^n$ is proportional to $n\rho_n=\rho_\varepsilon$. Using  Theorem~\ref{Thm:alaNeishtadt} we will get interpolation error of the order of $O(\exp(-c/\rho_\varepsilon))$. In order to achieve longer stability times we would like to reduce $\rho_\varepsilon$. Then we also get sharper estimates for  changes in actions. On the other hand,  there are two factors which limit our ability to decrease $\rho_\varepsilon$.
    \begin{itemize}
        \item[{(1)}] In the Hamiltonian $H_m$ the integrable part (represented by $h_n$
        which depends on $I$ only) is to dominate $w_{n,m}\sim n\varepsilon$.
    The integrable part
    is approximately quadratic in actions, i.e. $h_n\sim n\rho_n^2$. Therefore
    we need $n\rho_n^2\gg n\varepsilon$ for all $n<N_\varepsilon$ or equivalently
    \[
    \rho_\varepsilon\gg \varepsilon^{1/2}N_\varepsilon.
    \]
    
        \item[{(2)}] The sizes of $\rho_n$ are to be sufficiently large to ensure 
     that the balls $B(I_*,\rho_n)$ cover all actions. It is sufficient to assume
    \[
    \rho_\varepsilon\gg N_\varepsilon^{-1/d} \, .
    \] 
  
    \end{itemize}     
The sharpest bounds are achieved
when these two restrictions are of the same order in $\varepsilon$. In particular we can choose
\(
N_\varepsilon=\varepsilon^{-d/2(d+1)}.
\)
\end{remark}

\section{Interpolating vector fields\label{Se:interpolatingVF}}

The interpolating vector fields were originally introduced in \cite{GelfreichV18} and used to approximate  dynamics of a near-the-identity map by the
flow of a vector field $X_m$ obtained by taking a weighted average of several consecutive iterates of
the map.
In this paper we use a similar construction based on the Newton interpolation scheme which uses the forward orbit $x_0,\ldots,x_m$ for the construction of $X_m(x_0)$.
We will show that this interpolation scheme is sufficiently accurate
for our proof of  the Nekhoroshev theorem. 

Let us describe the construction of an interpolating vector field $X_m$.
Let ${U} \subset \R^s$ be an open  domain,
 $f: {U} \to\R^s$  a real analytic function
 and $m\in \N$.
Suppose that there is a subset
$ U_0\subset U$
such that $f^k( U_0)\subset U$ for $0\le k\le m$.
Then the iterates $x_k=f^k(x_0)$ are defined for all $k\le m$ and  $x_0\in U_0$.  
There is a unique polynomial $P_m(t;x_0)$ 
of degree $m$ in $t$
such that $P_m(k;x_0)=x_k$ for $0\le k\le m$.
The interpolating vector field is defined
by
\[
X_m(x_0)=\frac{\partial P(0;x_0)}{\partial t}.
\]
The polynomial $P_m$ can be obtained
with the help of the 
Newton finite-difference interpolation scheme.
Consider the following finite differences:
\begin{equation}\label{Eq:finite-diff}
	\Delta_0(x)=x,
	\qquad 
	\Delta_{k}(x)=\Delta_{k-1}(f(x))-\Delta_{k-1}(x),
	\ k\ge 1.
\end{equation}
Then the Newton interpolating polynomial with equally spaced data points and
step $h=1$ takes the form
$$
P_m(t;x_0)=x_0+\sum_{k=1}^m\frac{\Delta_k(x_0)}{k!}t(t-1)\dots(t-k+1).
$$
Differentiating $P_m(t;x_0)$ with respect to $t$ at $t=0$, we get
\begin{equation}\label{Eq:interpol_VF}
	X_m(x_0)=\sum_{k=1}^m\frac{(-1)^{k-1}}{k}\Delta_k(x_0).
\end{equation}

\begin{remark}
We  say that $X_m$ is obtained by application of a {\em discrete averaging\/}
procedure to the map $f$ as the sum in 
\eqref{Eq:interpol_VF}
is a weighted average of $x_0,\ldots,x_m$. Namely
\begin{equation}\label{Eq:interpol_VF_p_nk}
	X_m(x_0)=\sum_{k=0}^m p_{mk}f^k(x_0)
\end{equation}
where  the coefficients $p_{mk}$ do not depend on the map and can be found explicitly:
$p_{m0}$ is the harmonic number and for $k>1$
\[
p_{mk}=(-1)^{k+1}\frac{m+1-k}{k(m+1)}\binom{m+1}{k}.
\]
We skip the derivation of these coefficients.
\end{remark}

This construction  can be applied  not only to a single map $f$ but also to a family of maps.
In the case of a tangent to the identity family, we can use the interpolation
procedure to recover coefficients of a formal embedding into a formal vector field.

Let us state our claim more formally. Let $ U\subset\R^s$ be an open set.
Suppose that $f_\mu: U\to \mathbb R^s$ is an analytic family tangent to the identity.
In other words, the maps are defined on $U$ for $|\mu|\le \mu_0$
and $f_0$ is the identity map, $f_0(x)=x$ for all $x\in U$.
Since the set $ U$ is open, for every $x_0\in  U$
and every $m\in\mathbb N$ we can find $\mu_m(x_0)>0$
such that $x_k=f_\mu^k(x_0)\in U$ for all $|\mu|\le \mu_m(x_0)$
and $|k|\le m$. Then the interpolating vector field $X_m$ is an analytic
function of $x$ and $\mu$ in a neighbourhood of $x=x_0$ and $\mu=0$.
The following lemma shows that the Taylor expansion of the right-hand side of
\eqref{Eq:interpol_VF} in powers of $\mu$ coincides with the
formal vector field up to the order $m$.

\begin{lemma}[formal interpolation] \label{mjet1}
There is a unique sequence of analytic functions $g_k:U\to\R^s$, $k\in\N$, such that 
for every $m\in\N$ and every $x\in  U$
\[
f_\mu(x)-\Phi^1_{G_m}(x)=O(\mu^{m+1})
\]
where $\Phi^1_{G_m}$ is the time one-map of the vector field $G_m=\sum_{k=1}^m\mu^kg_k$.
Moreover, the interpolating vector field \eqref{Eq:interpol_VF} of order $m$
satisfies 
\[
X_m(x)=G_m(x)+O(\mu^{m+1}).
\]
\end{lemma}

\begin{proof}
Let $x\in U$. Then $f_\mu(x)=x+\sum_{k=1}^\infty \mu^k f_k(x)$. The radius of convergence may depend on $x$.

For any sequence of coefficients $g_k$,
 the time $t$ map of the vector field $G_m$ is an analytic function of $\mu$ in a 
 neighbourhood of $x$ provided $t$ is sufficiently small,
 \[
 \Phi^t_{G_m}(x)=\sum_{k=0}^\infty \mu^k a_{m,k}(x,t).
 \]
The flow is a solution of the initial value problem 
 \[
 \partial_t  \Phi^t_{G_m}(x) = G_m(\Phi^t_{G_m}(x)),\qquad  \Phi^0_{G_m}(x)=x.
 \]
 The initial condition implies $a_0(x,0)=x$.
Since the series $G_m$  starts with $k=1$ the differential equation implies that  $\partial a_0(x,t)=0$ for all $t$.
Consequently, $a_0(x,t)=x$ for all $t$.

The initial condition  implies $a_k(x,0)=0$ for $k\ge1$,
 and the differential equation with $t=0$ implies that $\partial_t a_k(x,0)=g_k(x)$
 for $1\le k\le m$.
 Collecting the terms of the first order in $\mu$ we get
 \[
 \partial_t a_1(x,t)=g_1(x).
 \]
 Consequently $a_1(x,t)=g_1(x) t$. 
 Using induction in $k$ and collecting Taylor coefficients of order $k$ in  $\mu$,
 it is not too difficult to prove that $a_k(x,t)$ are defined uniquely and
   are polynomial in $t$ of order $k$ with coefficients depending on $x$.
   Moreover for $k\le m$ we get $a_k(x,t)=t g_k(x)+t^2 b_k(x,t)$, where $b_k$ depends on
   $g_1,\ldots,g_{k-1}$ only. Therefore for  $k\le m$ we can set
   \[
   g_k(x)=f_k(x)-b_k(x,1)
   \]
   and get $a_k(x,1)=f_k(x)$.
   
   It is easy to check that if we repeat the procedure with $m$ replaced by $m+1$ 
   the values of $a_k$ with $k\le m$ are not affected. Consequently, the series $g_k$
   are defined uniquely for all $k$.

The smooth dependence of a flow on its vector field implies that, for any $x \in {U}$,
\[
\Phi^t_{G_m}(x)=\sum_{j=0}^{m}\mu^j a_j(x,t)+\mu^{m+1}r_m(x,t,\mu)
\]
where $r_m$ is a bounded function in $\mathcal V=B_r(x) \times [0,m] \times \{|\mu|\le \mu_m(x)\}$
and $r>0$ depends on $x$ and $m$.
Consequently, for all $0\leq k \leq m$, $\Phi^k_{G_m}$ and $f_\mu^k$ have a common
$m$-jet in $\mu$ at the point $x$ and
\[
x_k = f_\mu^k(x)=\Phi^k_{G_m}(x_0)+\mu^{m+1}q_{m,k}(x_0,\mu)
\]
where $q_{m,k}$ is a bounded function on $\mathcal{V}$. Combining these two bounds
we obtain
\[
x_k=
\sum_{j=0}^{m}\mu^j a_j(x_0,k)+\mu^{m+1}\left(q_{m,k}(x_0,\mu) +r_m(x_0,k,\mu) \right).
\]
The  interpolation by
a polynomial of degree $m$ is exact on polynomials of degree $m$. Consequently,
\[
P_m(t)=\sum_{j=0}^{m}\mu^j a_j(x_0,t)+\mu^{m+1} R_m(x_0,t)
\]
where  $R_m(x_0,t)$ is the polynomial of degree $m$ in $t$ which interpolates  the points $q_{m,k}(x_0,\mu)
+r_m(x_0,k,\mu)$ with the node $t=k$ and $0\leq k \leq m$.
Taking the derivative at $t=0$ we get
\[
X_m(x)=\sum_{j=0}^{m}\mu^j\dot a_j(x,0)+\mu^{m+1}\dot R_m(x_0,0)
=
G_m(x)+\mu^{m+1}\dot R_m(x_0,0),
\]
where, by \eqref{Eq:interpol_VF_p_nk}, we get 
\[
\dot R_m(x,0)=\sum_{k=0}^m p_{mk} \left(q_{m,k}(x,\mu) +r_m(x,k,\mu) \right).
\]
Hence $\dot R_m(x,0)$  is bounded on $\mathcal{V}$
and
 $X_m(x) = G_m(x)+O(\mu^{m+1})$. 
\end{proof}

\goodbreak

\section{Embedding a symplectic near-the-identity map into an autonomous Hamiltonian flow\label{Se:NeishtadtThm}}

Suppose that a symplectic map $f$ is $\epsilon$-close to the identity
on a complex $\delta$-neighbourhood of $D_0\subset \C^{2d}$.
The following theorem shows that if the ratio $\delta/\epsilon$ is sufficiently
large  then an interpolating vector field of  optimal order
$m\sim \delta/\epsilon$ provides exponentially accurate
approximation for the map $f$. In contrast to the classical
result of Neishtadt \cite{Neishtadt84} our theorem provides
explicit expressions for all constants. 
Therefore our theorem can be applied not only to members of a near-the-identity
family (where we are able to decrease $\epsilon$ when necessary) but to an
individual map as well. This subtle difference will play the key role in our
proof of the Nekhoroshev theorem.

Let $D_0\subset \mathbb C^{2d}$ 
and $D$ be a $\delta$-neighbourhood of $D_0$.
Suppose that a symplectic map $f:(p,q)\mapsto (P,Q)$
admits a generating function of the form
\begin{equation} \label{nearId_gf}
G(P,q)= Pq + {S}(P,q),
\end{equation}
i.e., the map is defined implicitly by the equations
$$
p=P+\frac{\partial S}{\partial q}(P,q), \quad Q=q+\frac{\partial S}{\partial P}(P,q).
$$
We assume that $S$ has an analytic continuation onto $D$ and we use
\begin{equation*}
\epsilon = \bigl\|
\nabla {S}\bigr\|_{D} 
=\sup_{(P,q)\in  D}\max\bigl\{\,|P-p|,|Q-q|\,\bigr\}
\end{equation*}
to characterise the closeness of $f$ to the identity.
Our definition
is slightly different from 
the traditional one where the supremum is
taken over the domain of the map 
while we use the domain of its generation function.
Our choice slightly simplifies analysis of transitions between 
a symplectic map and its generating function.

In the following theorem we use the infinity norm for vectors
and supremum norms for functions. We use $\left\lfloor \cdot \right\rfloor$ 
to denote the integer part of a number.

\begin{thm}\label{Thm:alaNeishtadt}
If $\displaystyle m = \left\lfloor \frac{\delta}{6\mathrm e\,\epsilon}-d
\right\rfloor\ge1$ and $X_m$ is the interpolating vector field
\eqref{Eq:interpol_VF} of order $m$, then  $\|X_m\|_{D_1}\le 2\epsilon$ and 
\[
\|\Phi_{X_{\mathrm{m}}}-f \|_{D_0}\le 
3\;\mathrm e^{d+1}\epsilon \exp\left(-\delta /(6\mathrm e\,\epsilon)\right),
\]
where $D_1$ is the $\frac{\delta}{2}$-neighbourhood of $D_0$.
Moreover there is a Hamiltonian  vector field
$\hat X_{m}$ such that 
\begin{equation*}
    \|\hat X_{m}- X_{m}\|_{D_1}
\le 
4\;\mathrm e^{d+1}\epsilon\exp\left(-\delta/(6\mathrm e \epsilon)\right),
\end{equation*}
and
\[
\|\Phi_{\hat X_{\mathrm{m}}}-f \|_{D_0}\le 
5\;\mathrm e^{d+1}\epsilon \exp\left(-\delta /(6\mathrm e\,\epsilon)\right).
\]
Moreover, $\|\hat X_m\|_{D_1}\le 4\epsilon$ and
\begin{equation}\label{Eq:XmLeadingOrder}
    \|\hat X_{m}-\mathrm J \nabla S\|_{D_1}\le 
\frac
{c\epsilon^2}{\delta} 
\;
\end{equation}
where $c=17(d+3)^2$ and $\mathrm J$ is the standard symplectic matrix.
\footnote{The constant in the estimate \eqref{Eq:XmLeadingOrder} is 
not optimal. It is obtained using the interpolation of the first order and  can be improved using a more accurate approximation for $H_m$.}
\end{thm}

\begin{remark} \label{remark_cota}
We also prove the following statements about interpolating vector fields. Under the assumption of the theorem
for every $m$ such that
\[
1\le m < \frac{\delta}{6 \epsilon}-d
\]
the following inequalities hold: $\|X_m\|_{D_1}\le 2\epsilon$, $\|\hat X_m\|_{D_1}\le 4\epsilon$,
\begin{align*}
\|\Phi_{X_m}-f \|_{D_0} &\le 3 C_m^{m} \epsilon^{m+1} ,
\\
\|\Phi_{\hat X_m}-f \|_{D_0} &\le 5 C_m^{m} \epsilon^{m+1},
\\
\|\hat X_{m}- X_{m}\|_{D_1}
&\le 
4C_m^{m}\epsilon^{m+1},
\end{align*}
where
\(
C_m= \frac{6(m+d)}{\delta}.
\)
These bounds show that $X_m$ provides an  embedding of the map into a flow with $O(\epsilon^{m+1})$ error. The exponential bound is obtained by choosing $m$ to minimize the
error bound. The best approximation of the map by an interpolating flow 
is achieved when $m\approx \delta/6\mathrm e\epsilon$.
This step is  possible due to the explicit control of the constants in 
the error bounds.
\end{remark}

\begin{proof}
We consider the map 
 $f$ as a member of a 
family of symplectic maps $f_\mu$ defined implicitly by the generating function 
$$
G_\mu(P,q) = Pq+\mu S(P,q)
$$
where $\mu$ is a complex parameter.  When $\mu=1$ the map $f_\mu$ coincides with~$f$. When $\mu=0$ the map $f_\mu$ is the identity. Therefore this family interpolates between
$f$ and the identity map  $\xi:(p,q)\mapsto (p,q)$.
Obviously the  function $G_\mu$ is analytic in the same domain $D$ as the function $S$. 

First we are going to prove that if $x_0\in D_1$, $k\in\N$ and
\[
|\mu|\le \mu_k=\frac{\delta}{2\epsilon(k+d)},
\]
then 
 $x_k=f_\mu^k(x_0)\in D$ and 
\(
\bigl|x_k-x_{k-1}\bigr|\le |\mu| \epsilon.
\)
Indeed, let $x_k=(p_k,q_k)$, the trajectory is defined by 
 the system
\begin{equation}\label{Eq:orbit}
\left\{
\begin{aligned}
p_{k-1} &= p_k + \mu \frac{\partial{S}}{\partial q}(p_k,q_{k-1}),\\
q_k &= q_{k-1} + \mu \frac{\partial{S}}{\partial P}(p_k,q_{k-1}).
\end{aligned}
\right.
\end{equation}
In order to find $p_k$ we need to solve the first equation. Then
 we substitute the solution into the second one. 
The  Implicit Function Theorem~\ref{Thm:IFT}
implies that the system with $k=1$ has a solution  $ (p_1,q_1)\in D$.
We continue with the help of finite induction in $k$. 
Suppose that $k\in\N$ and the first $k-1$ iterates of $x_0$ belong to $D$
provided $|\mu|\le \mu_{k-1}$.
Then the system implies that 
\[
|x_{k-1}-x_0|\le \sum_{j=1}^{k-1}|x_j-x_{j-1}|\le (k-1)|\mu|\epsilon.
\]
Let $
r_{k-1}=\frac12\delta -(k-1)|\mu|\epsilon
$.
Since  $x_0\in D_1$, the set $D$ contains the $\frac{\delta}{2}$-neighbourhood of 
$x_0$ and, consequently, the  ball $B_{r_{k-1}}(x_{k-1})\subset D$.
Taking into account the definition of $\mu_k$ we get that for  $|\mu|\le \mu_k$
\[
r_{k-1}=\tfrac12\delta -(k-1)|\mu|\epsilon=(k+d)\mu_k\epsilon-(k-1)|\mu|\epsilon
\ge (d+1)|\mu|\epsilon.
\]
We see that the assumptions of the 
 Implicit Function Theorem~\ref{Thm:IFT}
are satisfied by the first line of the system and consequently 
it defines $p_k$ as a function of $(p_{k-1},q_{k-1})$. It is not too difficult to 
check that  $x_k=(p_k,q_k)\in  D$.

\medskip

Now we can study interpolating vector fields for the map $f_\mu$.
First we are to find  upper bounds for  the finite differences. We
introduce the following notation: for a function $g$ let $\Tf_f(g)=g\circ f$ and $\I(g)=g$.
 We note that
$$
	(\I-\Tf_{f_\mu})^k\xi
	=(\I-\Tf_{f_\mu})^{k-1}(\xi-f_\mu)
	=
	\sum_{j=0}^{k-1}
	\binom{k-1}{j}
	(-1)^j(f_\mu^j-f_\mu^{j+1})
$$
(recall that $\xi$ stands for the identity map) and consequently
\begin{equation} \label{bound_D_k}
	\left\|(\I-\Tf_{f_\mu})^k\xi\right\|_{D_1}
	\le 
	\sum_{j=0}^{k-1} \binom{k-1}{j} \left\|f_\mu^j-f_\mu^{j+1}\right\|_{D_1}
	\le  
	2^{k-1}|\mu|\epsilon.
\end{equation}
Next we recall that the operator $\Tf_{f_\mu}-\I$ increases valuation in $\mu$,  consequently 
$\val_{\mu}((\I-\Tf_{f_\mu})^k\xi)\ge k$. 
Applying the MMP%
\footnote{The degree of the first non-zero monomial of the Taylor expansion in a variable
 $\zeta$ defines a valuation, that will be denoted by $\val_\zeta$, of the ring
 of formal series $\mathbb{C}[[\zeta]]$.  We will use the following simple statement of complex analysis:
 if $g$ is analytic in $\{\zeta \in \mathbb{C}, |\zeta|\leq \zeta_0\}$, $\zeta_0>0$, and $\val_\zeta(g)\geq k$ then it follows from the maximum modulus principle (MMP) that 
 $|g(\zeta)|\leq (|\zeta|/\zeta_0)^{k} \max_{|\zeta|=\zeta_0 } |g(\zeta)|$. }
for each $x \in D_1$ fixed, we obtain 
$$
\left\|(\I-\Tf_{f_\mu})^k\xi\right\|_{D_1}
\le
\frac{|\mu|^k}{\mu_k^k}\sup_{|\mu|=\mu_k}\left\|
(\I-\Tf_{f_\mu})^k\xi
\right\|_{D_1}
\le
 \frac{|\mu|^k2^{k-1}\epsilon}{\mu_k^{k-1}} . 
$$
Now we let $m \geq 1$ and consider the interpolating vector field \eqref{Eq:interpol_VF} written in the form
$$
X_{m,\mu}=-\sum_{k=1}^m\frac{1}{k}(\I-\Tf_{f_\mu})^k\xi.
$$
If  $|\mu|\le \mu_m$, it
is analytic in $D_1$ and admits the following upper bound 
\[
\|X_{m,\mu}\|_{D_1}\le\sum_{k=1}^m \frac1{k} \left\|(\I-\Tf_{f_\mu})^k\xi\right\|_{D_1}
\le \epsilon |\mu| \sum_{k=1}^m \frac1k\left( \frac{2|\mu|}{\mu_k} \right)^{k-1}.
\]
Using that $\mu_k\ge\mu_m$ for $k\le m$, we get that 
\begin{equation}
\|X_{m,\mu}\|_{D_1} 
\le \epsilon |\mu| \sum_{k=1}^m \frac1{k}
\left(\frac{2}{3}\right)^{k-1}
<\epsilon |\mu|\,\frac{3}{2}\log 3< 2\epsilon|\mu|
\label{norD1opt2}
\end{equation} 
for
$|\mu| \leq \mu_m/3$. 
If $\mu_m\ge3$, the domain of validity of the upper bound 
includes $\mu=1$. 
Since $X_m = X_{m,1}$ we conclude that 
\[
\|X_{m}\|_{D_1}< 2 \epsilon.
\]
Now we consider 
$\Phi_{X_{m,\mu}}$, 
the time-one map  of the vector field $X_{m,\mu}$.
Equation \eqref{norD1opt2} implies that
\[
 \|X_{m,\mu}\|_{D_1}<\frac{2\epsilon\mu_m}{3}=\frac{\delta}{3(m+d)}\le\frac{\delta}{6}.
\]
Then the orbit of every point in $D_0$ remains in $D_1$ during one unit of time and
$$
\|\Phi_{X_{m,\mu}}-\xi\|_{D_0}\le \|X_{m,\mu}\|_{D_1}.
$$
Then
$$
\|\Phi_{X_{m,\mu}}-f_\mu \|_{D_0}\le \|\Phi_{X_{m,\mu}}-\xi\|_{D_0}+\|\xi-f_\mu \|_{D_0} \le \|X_{m,\mu}\|_{D_1}+\epsilon|\mu| =3\epsilon|\mu|.
$$
Lemma~\ref{mjet1} states that the Taylor expansion in $\mu$ of $\Phi_{X_{m,\mu}}$ matches the Taylor expansion of
$f_\mu$ up to the order $m$.
Then MMP implies that
\begin{equation}\label{Eq:Phi-f_mu}
\|\Phi_{X_{m,\mu}}-f_\mu \|_{D_0}  
\le   \epsilon\mu_m\, \frac{ 3^{m+1}|\mu|^{m+1}}{\mu_m^{m+1}} .
\end{equation}
Substituting $\mu=1$ we obtain
\begin{equation}\label{Eq:Phi-f}
    \|\Phi_{X_m}-f \|_{D_0} 
  \le   3\epsilon\, \frac{ 3^{m}}{\mu_m^{m}} 
  =
    3\epsilon  \left( \frac{6(m+d)\epsilon}{\delta}\right)^{m} .
\end{equation}
The right hand side depends on $m$ and takes the smallest values 
somewhere near  $m=\left\lfloor M_\epsilon \right\rfloor$ where
\begin{equation}\label{Eq:mepsilon}
   M_\epsilon=\frac{\delta}{6 \mathrm e\epsilon}-d.
\end{equation} 
We note that for $m\le M_\epsilon$ we get that
\[
\mu_m=\frac{\delta}{2\epsilon(m+d)}\ge
\frac{\delta}{2\epsilon(M_\epsilon+d)}= 3\mathrm e>6
\]
and consequently the inequality \eqref{Eq:Phi-f} holds for all these values of $m$.
We notice that for $m=\left\lfloor M_\epsilon \right\rfloor$ we have
$$
\frac{6(m+d) \epsilon}{\delta}\le \mathrm e^{-1}
$$
and consequently
$$
\|\Phi_{X_m}-f \|_{D_0} 
\le 3 \epsilon \, \mathrm e^{-m} 
\le 3\epsilon\, \mathrm e^{1+d} \exp
\left(-
\frac{\delta}{6\mathrm e \epsilon}
\right).
$$
The interpolating vector fields $X_{m,\mu}$  are rarely Hamiltonian. 
On the other hand, the formal interpolating vector field is Hamiltonian.
Although this property is known, we give a simple
proof in Appendix~\ref{appendix:Hperiodic}.
Then Lemma~\ref{mjet1} implies that $\hat X_{\mu,m}$, the Taylor polynomial in $\mu$ of degree $m$ for $X_{m,\mu}$, is
Hamiltonian.
For example, the interpolating vector field of the first order
is given by 
\[
X_{1,\mu}(p,q)=f_\mu(p,q)-(p,q).
\]
In general, there is no reason for this vector field
to be Hamiltonian. On the other hand,  its Taylor polynomial of degree one,
\[
\hat X_{1,\mu}=\left.\frac{\partial X_{1,\mu}}{\partial\mu}\right|_{\mu=0}\mu
=\left.\frac{\partial f_{\mu}}{\partial\mu}\right|_{\mu=0}\mu,
\]
is Hamiltonian. In order to find the corresponding Hamiltonian function we recall that
the map $f_\mu:(p,q)\mapsto (p_1,q_1)$ is defined implicitly by the system
\eqref{Eq:orbit} with $k=1$ (we assume $(p_0,q_0)=(p,q)$ are independent of $\mu$).
Differentiating the system with respect to $\mu$
at $\mu=0$ and  using that $p_1=q$ and $q_1=q$ for $\mu=0$,
we get
\[
\left\{
\begin{aligned}
0 &= \frac{\partial p_1}{\partial\mu} + \frac{\partial{S}}{\partial q}(p,q),\\
\frac{\partial q_1}{\partial \mu} &= 0 + \frac{\partial{S}}{\partial p}(p,q).
\end{aligned}
\right.
\]
We conclude that
\begin{equation} \label{H1}
\hat X_{1,\mu}(p,q)= \mu \left( 
- \frac{\partial{S}}{\partial q}(p,q),\frac{\partial{S}}{\partial p}(p,q)\right).
\end{equation}
We see that the vector field $\hat X_{1,\mu}$ is Hamiltonian
with the Hamiltonian function $H_{1,\mu}=\mu S$.

In order to estimate $\hat{X}_{m,\mu}$ for $m\le M_\epsilon$,
we notice that equation  \eqref{norD1opt2} implies
that $\|X_{m,\mu}\|_{D_1} \le 2 \epsilon \mu_m/3$ 
for  $|\mu| \leq  \mu_m/3$. Then for $k\le m$
\[ \frac{1}{k!} \left\|\partial_\mu^k X_{m}\bigl|_{\mu=0}\right\|_{D_1}
\le 2 \epsilon \left(\frac{3}{\mu_m}\right)^{k-1}
\]
and
\[
\bigl\|\hat X_{m,\mu}\bigr\|_{D_1}\le \sum_{k=1}^m \frac{1}{k!} \left\|\partial_\mu^k X_{m,\mu}\bigl|_{\mu=0}\right\|_{D_1} |\mu|^k
\le
2 \epsilon |\mu| \sum_{k=1}^m  \left(\frac{3 |\mu|}{\mu_m}\right)^{k-1}.
\]
For $|\mu| \le \mu_m/6$ we get
\begin{equation} \label{norD1optH2}
 \|\hat X_{m,\mu}\|_{D_1} \le  2 \epsilon |\mu| \sum_{k=1}^{\infty} \left( \frac12\right)^{k-1} \le 4 \epsilon |\mu|.
\end{equation}
Since $ \mu_m>6$
 we can substitute $\mu=1$
 to obtain 
\[ \|\hat X_m\|_{D_1} \le 4 \epsilon. 
\]
Then we repeat the previous
arguments using $\hat X_{m,\mu}$ instead of $
X_{m,\mu}$ and 
  the upper bound
\eqref{norD1optH2} instead of \eqref{norD1opt2}. 
The equation \eqref{norD1optH2} implies that  $ \|\hat X_{m,\mu}\|_{D_1} \le 2\epsilon\mu_m/3=\frac{\delta}{3(m+d)}<\frac{\delta}{6}$ for all  $|\mu|\le \mu_m/6$.
Then repeating the previous arguments we get
$\|\Phi_{\hat{X}_{m,\mu}}-f_\mu \|_{D_0}\le 5\epsilon|\mu|$. 
Then using the MMP we get
$$
\|\Phi_{\hat{X}_{m,\mu}}-f_\mu \|_{D_0}  
\le   \frac{5\epsilon\mu_m}{6}\, \frac{ 6^{m+1}|\mu|^{m+1}}{\mu_m^{m+1}} .
$$
Substituting $\mu=1$ we get
$$
\|\Phi_{\hat{X}_m}-f \|_{D_0}  
\le   5 \epsilon \left( \frac{ 6(m+d) \epsilon}{\delta} \right)^{m}.
$$
In particular, for  
$m=\left\lfloor M_\epsilon\right\rfloor $ 
\[
\|\Phi_{\hat X_{\mathrm{m}}}-f \|_{D_0}\le 
5\;\mathrm e^{d+1}\epsilon \exp\left(-\delta /(6 \mathrm e\,\epsilon)\right).
\]
Since $\hat X_{m,\mu}$ is the Taylor polynomial of $X_{m,\mu}$ of order $m$ we can use the standard bound
for the remainder (the radius of convergence is $\mu_m/3$, the bound is for $\mu=1$, $\mu_m>6$):
\[
\left\|\hat X_{m}- X_{m}\right\|_{D_1}\le 
\frac{(2\epsilon\mu_m/3  )(3/\mu_m)^{m+1}}{1-3/\mu_m}
\le 4\epsilon  (3/\mu_m)^{m}
=4\epsilon\left( \frac
{6\epsilon(m+d)}{\delta}\right) ^{m}.
\]
Substituting
 $m=\left\lfloor M_\epsilon \right\rfloor$
we get 
\begin{equation}\label{Eq:hatX_m-X_m}
\|\hat X_{m}- X_{m}\|_{D_1}\le 
4\epsilon\,\mathrm e^{1+d}\exp\left(-\frac{\delta}{6\mathrm e \epsilon}\right).  
\end{equation}
We see that the interpolating vector field $X_m$ is exponentially
close to a Hamiltonian one.

In order to complete the proof we have 
 to show that the vector fields $\hat X_m=\hat X_{m,1}$
 are close to $S$ for all $m\le M_\epsilon$.
 Using the upper bounds for the derivatives we get
\[
\|\hat X_{m}- \hat X_{1}\|_{D_1}\le 
\sum_{k=2}^m
\frac{1}{k!} \left\|\partial_\mu^k X_{m}\bigl|_{\mu=0}\right\|_{D_1}
\le 
2\epsilon\sum_{k=2}^m
\left(\frac{3}{\mu_k}\right)^{k-1}.
\]
It is not too difficult to check that the sequence 
$\left(\frac{3}{\mu_k}\right)^{k-1}$
is monotone decreasing for $k\le m\le M_\epsilon$. Therefore
\[
\|\hat X_{m}- \hat X_{1}\|_{D_1}\le 
\frac{6\epsilon}{\mu_2}
+2\epsilon\sum_{k=3}^m
\left(\frac{3}{\mu_k}\right)^{k-1}
\le 
\frac{6\epsilon}{\mu_2}+\frac{18\epsilon \,M_\epsilon}{\mu_3^2}
.
\]
Recalling the definitions of $\mu_k$ we get
\[
\|\hat X_{m}- \hat X_{1}\|_{D_1}\le 
\frac{12\epsilon^2}{\delta}(d+2)
+\frac{12\epsilon^2  }{ \mathrm e \delta }(d+3)^2
\le 
\frac{17\epsilon^2  }{ \delta }(d+3)^2
.
\]
We get the desired estimate as $\hat X_1=\mathrm J \nabla S$.
\end{proof}

\section{Interpolating flow near a fully resonant torus \label{Sect:Nekhoroshev}}

This section contains the proof of Theorem~\ref{Thm:longstab}.
We recall that $I_*\in B_R$
corresponds to a fully resonant torus, i.e., $n\omega(I_*)\in\Z^d$ for some natural
$n<N_\varepsilon$. For the rest of this section we assume that
\[
\gamma_0^2=\frac{18 d\|a\|}{\nu}
\qquad\text{and}\qquad
r_0^2={\frac{\nu}{6 d\|h_0''\|}}\,.
\]
We also fix $\gamma\ge\gamma_0$. We will reduce $\varepsilon_0$
when necessary.

\subsection{Exponentially accurate interpolation \label{Se:intrepol_at_Torus}} 

Theorem~\ref{Thm:longstab} establishes estimates for stability times for real initial conditions. On the other hand, in order to use Theorem~\ref{Thm:alaNeishtadt} we need a bound of the map $f_\varepsilon^n$
in a complex neighbourhood of its real domain.
As a first step we check that the map satisfies
the assumptions of  Theorem~\ref{Thm:alaNeishtadt}  in 
\begin{equation}\label{Eq:DI*}
\mathcal D(I_*)=\left\{\,
(I,\varphi)\in\mathbb C^{2d}: |I-I_*|<2\rho_n,|\operatorname{Im}(\varphi)|<r/2\,\right\},
\end{equation}
a complex neighbourhood of the domain $\mathcal D_0(I_*)$ defined  in \eqref{Eq:D(I*)}.
Let $(I_0,\varphi_0)\in \mathcal D(I_*)$.
Applying Lemma~\ref{Le:apriori} recursively to check that the previous iterates do not leave the domain $\mathcal D_F$,
we conclude that
\[
\begin{aligned}
&|I_k-I_*|\le |I_0-I_*|+C_1k \varepsilon\le 2\rho_n+C_1k\varepsilon<\sigma,\\
&|\operatorname{Im}(\varphi_k-\varphi_0- k \omega(I_0))|
\le C_2 k^2 \varepsilon< \frac{r}{4},
\end{aligned}
\] 
while $k$ is not too large.
Then
\[
\begin{split}
|\operatorname{Im}(\varphi_k-\varphi_0)|
\le 
\frac r4 + k |\operatorname{Im}( \omega(I_0))|
\le
\frac r4 + k \|\omega'\|\,|\operatorname{Im}(I_0)|
\le \frac{r}4+C_3k\rho_n<\frac r2.    
\end{split}
\]
Here we  use the constant $C_3=2 \|\omega'\|$
and assume that
\begin{equation}\label{Eq:C1-3}
C_1 k\varepsilon <\frac\sigma2,\qquad
C_2 k^2 \varepsilon<\frac r4,\qquad C_3k\rho_n<\frac r4,
\qquad 4\rho_n<\sigma.
\end{equation}

Recalling our choice of $\rho_n=\rho_\varepsilon/n$, $\rho_\varepsilon=\gamma\varepsilon^{1/2(d+1)}$ and $n<N_\varepsilon=\varepsilon^{-d/2(d+1)}$, see \eqref{Eq:rho_eps} and \eqref{Nmax_nearid},
it is easy to check that the inequalities \eqref{Eq:C1-3} are satisfied
for $k \leq n$
provided $\varepsilon<\varepsilon_0$ with a sufficiently small $\varepsilon_0$
(independent of $n$).
Then  the first $n$ iterates $(I_n,\varphi_n)=F_\varepsilon^n(I_0,\varphi_0)$
are well defined for initial conditions in $\mathcal D(I_*)$.
Since $|\omega(I_0)-\omega(I_*)|\le \|\omega'\| 2\rho_n$
we get from Lemma~\ref{Le:apriori}  that
\begin{equation} \label{n_it_estimates}
    |I_n-I_0|\le C_1 n\varepsilon,
\qquad
|\varphi_n-\varphi_0-n\omega(I_*)|
\le C_2 n^2\varepsilon + C_3 n\rho_n.
\end{equation}
To study the dynamics in $\mathcal D(I_*)$ we 
introduce  translated and scaled actions $J$
 with the help of the equality
\begin{equation} \label{scaling}
I=I_*+\rho_{n} J.
\end{equation}
Let $\hat{f}_\varepsilon^n$ denote the map \eqref{mapFn}  expressed in the new coordinates. 
It can be written in the form
$\hat f_{\epsilon}^n:(J,\varphi)\mapsto (\bar J,\bar \varphi)$,
\begin{equation}\label{Eq:hatfepsn}
\left\{
\begin{aligned}
	\bar J&=J+\rho_n^{-1}\varepsilon\sum_{k=0}^{n-1} a(I_{k},\varphi_{k}),
 \\
\bar \varphi&=\varphi+\sum_{k=0}^{n-1}(\omega( I_{k})-\omega(I_*))+\varepsilon\sum_{k=0}^{n-1} b(I_{k},\varphi_{k}),
\end{aligned}	
\right.
\end{equation}
and $(I_k,\varphi_k)=F_\varepsilon^k(I_0,\varphi_0)$  denote
iterates of $(I_0,\varphi_0)=(I_*+\rho_n J,\varphi)$ under the original map \eqref{Eq:feps}.
The map $\hat f_\varepsilon^n$  is $\epsilon_n$-close
to the identity on $\mathcal D(I_*)$
where
\[
\begin{split}
\epsilon_n&=
\sup_{|J|<2,|\operatorname{Im}(\varphi)|<r/2}
\max\{\,|\bar J-J|,|\bar \varphi-\varphi|\,\}
\\
&=\sup_{(I_0,\varphi_0)\in \mathcal D(I_*)}
\max\{\,\rho_n^{-1}|I_n-I_0|,|\varphi_n-\varphi_0-n\omega(I_*)|\,\}.
\end{split}
\]
Then, the estimates \eqref{n_it_estimates} show that
\begin{equation} \label{dtoId}
\epsilon_n\le 
\max\left\{ C_1
\rho_n^{-1}n\varepsilon,\,  C_3 n\rho_n+C_2 n^2\varepsilon
\right\} =
n\rho_n\max\left\{ C_1
\frac{\varepsilon}{\rho_n^2},\,  C_3 +C_2 \frac{n\varepsilon}{\rho_n}
\right\}
.
\end{equation}

Since
\[
\frac{\varepsilon}{\rho_n^2}=\frac{n ^2\varepsilon}{\rho_\varepsilon^2}
=
\frac{n^2 \varepsilon^{d/(d+1)}}{\gamma^2}<
\frac{N_\varepsilon^2\varepsilon^{d/(d+1)}}{\gamma^2}=\frac{1}{\gamma^2}
\]
and
\[
\frac{n\varepsilon}{\rho_n}=\frac{n^2\varepsilon}{\rho_\epsilon}<
\frac{\rho_\varepsilon}{\gamma^2}=\frac{\varepsilon^{1/2(d+1)}}{\gamma},
\]
there is a positive constant  $C_4$ such that  
\begin{equation} \label{dtoId_Lochack}
\epsilon_n\le C_4 n\rho_n
=
\gamma C_4 \varepsilon^{1/2(d+1)}
.
\end{equation}

 The Implicit Function Theorem~\ref{Thm:IFT} can be applied to the first component of
 \eqref{Eq:hatfepsn}
 to show that $J$ is an analytic function of $\bar J$ and $\varphi$ on the set
\[
\hat D=\left\{\,(J,\varphi)\in\mathbb C^d:|J|\le 2 -(d+1)\epsilon_n,\;
|\operatorname{Im}(\varphi)|<r/2\,
\right\}.
\]
Then the expression for $J$
can be substituted into the second component of \eqref{Eq:hatfepsn}
to express $(J,\bar \varphi)$ as a function of $(\bar J,\varphi)$. Since the map
is symplectic, then according to Appendix~\ref{Se:genfun} there is a function
$S_n$ such that
\begin{equation}\label{Eq:Sn_map}
\bar J-J=-\frac{\partial S_n}{\partial \varphi}(\bar J,\varphi),
\qquad
\bar \varphi-\varphi=
\frac{\partial S_n}{\partial \bar J}(\bar J,\varphi).
\end{equation}
The definition of $\epsilon_n$ implies 
\[
\|\nabla S_n\|_{\hat D}\le \epsilon_n.
\]
Let
\begin{equation} \label{delta_gamma0}
\delta=\tfrac12 \min\{1,r\}\qquad\text{and}\qquad
c_3=\frac{\delta}{C_4\gamma 6\mathrm e}.
\end{equation}
If $\epsilon_n\le 1/2(d+1)$, then $\hat D$ contains a complex $\delta$-neighbourhood
of the real set
\[
\hat D_0=B(0,1)\times \mathbb R^d.
\] If 
 $\epsilon_n\le \delta/6\mathrm e(d+1)$, then the map $\hat f_\varepsilon^n$
 satisfies the assumptions of Theorem~\ref{Thm:alaNeishtadt}
 which states that there is $m=m(\epsilon_n)\sim\epsilon_n^{-1}$ such that the
 time-one map of the  Hamiltonian vector  field $\hat X_m$ approximates $\hat
 f_\varepsilon^n$ with exponential accuracy:
\begin{equation}\label{Eq:exp_near_res}
\left\|\hat f_\varepsilon^n-\Phi_{\hat X_m}\right\|_{B(0,1)\times\mathbb R^d }
\le 5\mathrm e^{d+1}\epsilon_n \exp\left(-c_3 \, \varepsilon^{-1/2(d+1)}\right).    
\end{equation}
According to the theorem 
 $\hat X_m$
is close to the vector field  with the Hamiltonian function $S_n$.
We will analyse the Hamiltonian function of $\hat X_m$ in the next subsection.

\begin{remark} \label{Rk_scaling}
In this section we use the linear scaling \eqref{scaling} of the original action
coordinate $I$ by the factor $\rho_{n}$. This scaling is useful to enable a
direct application of Theorem~\ref{Thm:alaNeishtadt}. However, we can compute
$X_m$ using the iterates  of $f_\varepsilon^n$ in the original coordinates
$(I,\varphi)$ as the interpolation procedure commutes with linear changes of
variables.
\end{remark}

\subsection{Long term stability of actions} \label{Se:stability}

In the previous section we have established that $f_\varepsilon^n$, 
the lift of the map $F_\varepsilon^n$, is exponentially close to the time-one
map of the autonomous Hamiltonian vector field $\hat{X}_m$ in a small
neighbourhood of the fully resonant torus. In this section we will derive an
approximation for the corresponding Hamiltonian function $H_m$ and use its
properties to establish which  trajectories of the map are trapped inside this
neighbourhood for exponentially long times. 

According to
\eqref{Eq:XmLeadingOrder}, the interpolating Hamiltonian $H_m$ is close to
$S_n$, the generating function of the map $\hat
f_\varepsilon^n:(J,\varphi)\mapsto (\bar J,\bar\varphi )$. 
As a first step we derive the leading order approximation for $S_n$. Comparing
\eqref{Eq:hatfepsn} with \eqref{Eq:Sn_map} we conclude that $S_n$ is a solution
of the system of equations
\[
\begin{aligned}
\frac{\partial S_n}{\partial \varphi}(\bar J,\varphi)
&=
-\rho_n^{-1}\varepsilon\sum_{k=0}^{n-1}a(I_k,\varphi_k),
\\
\frac{\partial S_n}{\partial \bar J}(\bar J,\varphi)
&=n(\omega(I_n)-\omega(I_*))+
\sum_{k=0}^{n-1} (\omega( I_{k})-\omega(I_n))+\varepsilon\sum_{k=0}^{n-1} b(I_{k},\varphi_{k}).
\end{aligned}
\]
Note that the right hand side is expressed in terms of $(I_k,\varphi_k)=F_\varepsilon
^k(I_* + \rho_n J,\varphi)$. Since the value of $J$ can be expressed in terms of
$(\bar J,\varphi)$, $(I_k,\varphi_k)$ can also be expressed in terms of $(\bar
J,\varphi)$. In particular $I_n=I_*+\rho_n \bar J$. The symplecticity of $\hat
f_\varepsilon^n$ implies existence of a solution. We write it in the form
\begin{equation}\label{Eq:Snbounds}
S_n(\bar J,\varphi)
=h_n(\bar J)
+w_n(\bar J,\varphi).
\end{equation} 
The first term 
\[
h_n(\bar J) =n\rho_n^{-1}\left(h_0(I_*+\rho_n \bar  J)-h_0(I_*)-\rho_n \, 
\omega(I_*) \cdot \bar J
\right)
\]
represents the part independent of $\varphi$ and
comes from the explicit calculation performed with the help of the equality $\omega(I)=h_0'(I)$. The second term is expressed as an integral
\[
w_n(\bar J,\varphi)=
\int_{(0,0)}^{(\bar J,\varphi)}\sum_{l=1}^d (
v_l \; d\bar J_l-u_l \; d\varphi_l )
\]
where the index $l$ refers to components of the vectors,
the integral does not depend on the path connecting the end
points and 
\begin{equation} \label{uv}
\left.
\begin{aligned}
u(\bar J,\varphi) &=
  \rho_n^{-1}\varepsilon\sum_{k=0}^{n-1}a(I_k,\varphi_k), \\
v(\bar J,\varphi)&=
  \sum_{k=0}^{n-1} (\omega( I_{k})-\omega(I_n))+\varepsilon\sum_{k=0}^{n-1} b(I_{k},\varphi_{k}).
\end{aligned} \right.
\end{equation}
The equation~\eqref{Eq:Snbounds} suggests that
$S_n$ has the form of an integrable part plus a perturbative term, denoted by $h_n$ and $w_n$ respectively.
The integrable part is approximately quadratic. Indeed, 
$h_n'(0)=0$ and the strong convexity and smoothness of $h_0$ imply that 
\begin{equation} \label{convexityhn}
 \frac{\nu\,n\rho_n }{2} |\bar J|^2
\le 
h_n(\bar J)
\le \frac{\nu_2n\rho_n }{2} |\bar J|^2.
\end{equation}
where $\nu_2= d\|h_0''\|$.
In order to see that $h_n$ dominates 
$w_n$ outside 
a small neighbourhood of the origin, we look for an explicit bound for $w_n$
paying attention to the uniformity for all resonances.
The triangle inequality and \eqref{uv} imply that
$$
|u|\le C_1 n\varepsilon\rho_n^{-1}.
$$
Then using arguments of Lemma~\ref{Le:apriori} we get
\[
\begin{aligned}
|v|&\le \sum_{k=0}^{n-1}\|\omega'\|\,|I_n-I_k|
+\varepsilon n \|b\|\le
\frac{\varepsilon n^2}{2}\|\omega'\|\,\|a\|+\varepsilon n \|b\|
\le C_2 n^2 \varepsilon .
\end{aligned}
\]
Taking into account the periodicity arguments it is sufficient to  consider $w_n$ on the set $\hat D_0=B(0,1)\times[-1,1]^d$. 
Then 
\begin{equation} \label{wnbound}
\|w_n\|_{\hat D_0}\le 
dC_2 n^2\varepsilon+d C_1n\varepsilon\rho_n^{-1}.
\end{equation}
With our choice of  $\rho_n=\rho_\varepsilon n^{-1}$,
$\rho_\varepsilon=\gamma\varepsilon^{1/2(d+1)}$
and $n<N_\varepsilon=\varepsilon^{-d/2(d+1)}$, see \eqref{Eq:rho_eps} and \eqref{Nmax_nearid}, we get
\[
\begin{split}
\frac{\|w_n\|_{\hat D_0}}{n\rho_n}
&\le 
dC_2 n\varepsilon \rho_n^{-1}+d C_1\varepsilon\rho_n^{-2} 
\le
dC_2 \varepsilon^{ 1/2(d+1)}\gamma^{-1}+d C_1\gamma^{-2},
\end{split}
\]
where we used the bounds presented after equation~\eqref{dtoId}.
With our choice of $\gamma\ge \gamma_0$, the second term in the sum does not exceed $\nu/18$.
Decreasing $\varepsilon_0$ (if necessary) we get
\begin{equation} \label{cotawn}
\frac{\|w_n\|_{
\hat D_0}}{n\rho_n}
\le\frac{\nu}{9} .
\end{equation}
In this way we have got upper bounds for both terms in \eqref{Eq:Snbounds}.

According to Theorem~\ref{Thm:alaNeishtadt} the interpolating vector field
is Hamiltonian, $\hat X_m=\mathrm J\nabla H_m$, and the Hamiltonian
$H_m$ is close to the generating function $S_n$ due to the bound \eqref{Eq:XmLeadingOrder}.
The Hamiltonian $H_m$ can be obtained by integrating the vector field $\hat X_m=\mathrm J\nabla H_m$.
Since the map $\hat f_\varepsilon^n$ is periodic
in angles the vector field is also periodic. 
Moreover,  since the map is exact symplectic
and  in Theorem~\ref{Thm:alaNeishtadt}  $\hat X_m$ is obtained from a truncated expansion,
Appendix~\ref{Se:genfun} shows that 
the Hamiltonian is periodic.
Then we can restrict the integration to a fundamental domain in the angle variables
to get from \eqref{Eq:XmLeadingOrder} the inequality
\[
\|H_{m}-S_n\|_{
\hat D_0}\le 
\frac
{2c d\epsilon_n^2}{\delta} \le  \frac{2c d C_4^2n^2\rho_n^2}{\delta} = C_5 n^2\rho_n^2
\; .
\]
Decreasing $\varepsilon_0$ (if necessary) we get
 $C_5 n\rho_n= C_5\rho_\varepsilon<\frac{\nu}{18}$.
Then the equation \eqref{Eq:Snbounds} implies that
\[
\left\|H_{m}-
h_n\right\|_{
\hat D_0}
\le 
C_5 n^2\rho_n^2 +\|w_n\|_{
\hat D_0}< \frac\nu6 n\rho_n
\, .
\]

With the help of the bounds \eqref{convexityhn} 
we get that 
\begin{equation}\label{Eq:H_mBounds}
   \tfrac12  \nu |J|^{2}-\tfrac16 \nu
\le \frac{H_{m}(J,\varphi)}{n\rho_n}
\le 
\tfrac12 \nu_2 |J|^2+\tfrac16 \nu \, .
\end{equation} 
Suppose that at some point the  Hamiltonian $H_m(J,\varphi)< n\rho_nE_{\mathrm{max}}$ 
where
 \(E_{\mathrm{max}} = \tfrac 13\nu.\) The first inequality of \eqref{Eq:H_mBounds} implies that  
 \[
 \frac\nu2|J|^2< \frac{\nu}{3}+\frac\nu6=\frac\nu2 \, .
 \]
Since the Hamiltonian flow preserves $H_m$,
the whole trajectory of $(J,\varphi)$
is inside the  domain $|J|<1$.

Now let $E_0=\frac14\nu$.  Any point with $|J|<r_0$
belongs to  the set of initial conditions which
satisfy the inequality $H_m(J,\varphi)< n\rho_nE_0$. Indeed,
our choice of $r_0$ implies that
\[
\frac{H_{m}(J,\varphi)}{n\rho_n}
<
\frac{\nu_2}{2}r_0^2+\frac\nu6  = \frac{\nu}{4}.
\]
Unlike the Hamiltonian flow, the map does not
preserve the energy. Fortunately the change in the
energy after a single iterate is exponentially 
small:
\[
\begin{split}
M_\varepsilon&=\left\|H_m\circ \hat f_\varepsilon^n-H_m\right\|
_{
\hat D_0}=
\left\|H_m\circ \hat f_\varepsilon^n-H_m\circ \Phi_{\hat X_m}\right\|_{
\hat D_0}
\\&\le
\|H_m'\|_{
\hat D_1}\;\left\|\hat f_\varepsilon^n-\Phi_{\hat X_m}\right\|_{
\hat D_0}
= \|\hat X_m\| _{
\hat D_1}
\;\left\|\hat f_\varepsilon^n-\Phi_{\hat X_m}\right\|_{
\hat D_0}
\\&\le
20\epsilon_n^2\mathrm e^{d+1} \exp\left(-c_3 \varepsilon^{-1/2(d+1)}\right)
\end{split}
\]
where we use \eqref{Eq:exp_near_res} 
and the bound $\|\hat X_m\|_{\hat D_1}\le 4\epsilon_n$
of Theorem~\ref{Thm:alaNeishtadt}.
Here ${\hat D_1}$ is the $\frac\delta2$-neighbourhood of ${\hat D_0}$.
If we take an initial condition with $|J|<r_0$ 
then the initial energy is below $n\rho_n E_0$. 
We can be sure that the point remains inside the domain $|J|<1$ while the energy does not exceed 
$n\rho_n E_{\mathrm{max}}$. In this case  we can use the telescopic sum to see  that
\[
\begin{split}
H_m\circ \hat f_\varepsilon^{kn}(J,\varphi)-H_m(J,\varphi)&=
\sum_{j=1}^k
(H_m\circ \hat f_\varepsilon^{jn}(J,\varphi)
-H_m\circ \hat f_\varepsilon^{(j-1)n}(J,\varphi))
\\&
\le k M_\varepsilon.
\end{split}
\]
Then $H_m\circ \hat f_\varepsilon^{kn}(J,\varphi)< n\rho_nE_{\mathrm{max}}$ for all 
$k \le n\rho_n (E_{\mathrm{max}}-E_0)/M_\varepsilon$.
Consequently, the minimal number of iterates  of $F_\varepsilon$ needed to start
with an energy below $n\rho_nE_0$ and finish above $n\rho_nE_{\mathrm{max}}$ is larger than
\[
T_{\varepsilon}= \frac{n^2\rho_n\nu 
}{12 M_\varepsilon}
\ge
\frac{n^2\rho_n\nu}{240\epsilon_n^2\mathrm e^{d+1}}
\exp\left(c_3 \, \varepsilon^{-1/2(d+1)}\right)
\ge 
c_2 \exp\left(c_3 \, \varepsilon^{-1/2(d+1)}\right)
.
\]
We have proved that a trajectory 
with an initial condition $(I_0,\varphi_0)$
such that $|I_0-I_*|<r_0\rho_n$
has the property $|I_{kn}-I_*| < \rho_n$
for $0\le k n\le T_{\varepsilon}$.
We complete the proof of Theorem~\ref{Thm:longstab} by noting that between multiples of $n$
the changes in action variables are  controlled by Lemma~\ref{Le:apriori}
and do not exceed $C_1n\varepsilon$. Consequently $|I_k-I_*|<\rho_n$
for all $k<T_\varepsilon$.
This argument completes the proof of Theorem~\ref{Thm:longstab}.
\goodbreak

\section{Nucleus of a resonance \label{Se:Nucleus}}

Our proof of the exponential estimates for the stability times of the action variables
uses a covering of
the phase space by $\rho_n$-neighbourhoods of unperturbed fully resonant tori. Each of these tori is characterised by its frequency $\omega_*$ such that $n\omega_*\in\Z^d$ for some  $n<N_\varepsilon=\varepsilon^{-d/2(d+1)}$.
Therefore a fully resonant torus of period $n$ is included into the analysis
when  $\varepsilon$ becomes smaller than $n^{-2(d+1)/d}$ and eventually every fully resonant torus is used. In this section we show that every fully resonant torus has a small neighbourhood, which we call a {\em nucleus of the resonance},
where  the stability times are much longer than in the Nekhoroshev theorem.
Moreover the difference in stability times grows as $\varepsilon$ decreases
due to the presence of the factor $\varepsilon^{-1/2}$ instead of $\varepsilon^{-1/2(d+1)}$
in the exponent.

For the purpose of this analysis it is  convenient to rewrite the map
$F_\varepsilon:(I,\varphi)\mapsto (\bar I,\bar\varphi)$ with the help of a generating function
\[
S(\bar I,\varphi)=\bar I\cdot\varphi + h_0(\bar I)+\varepsilon s (\bar I,\varphi)
\]
where the function $s$ depends periodically on the angles $\varphi$.
Then the map is defined implicitly by the system
\begin{equation}\label{Eq:feps_crossform}
\left\{
\begin{aligned}
	\bar I&=I-\varepsilon \frac{\partial s}{\partial\varphi}(\bar I,\varphi),\\
	\bar \varphi&=\varphi+\omega(\bar I)+\varepsilon \frac{\partial s}{\partial\bar I}(\bar I,\varphi) \pmod1,
\end{aligned}	
\right.
\end{equation}
where $\omega(\bar I)=h_0'(\bar I)$.
When $\varepsilon=0$, these equations can be easily solved explicitly.
On the other hand, the geometric arguments of 
Appendix~\ref{Se:genfun}  
and Implicit Function Theorem \ref{Thm:IFT} can be used to show that every quasi-integrable map
$F_\varepsilon$ can be represented in this form.
 The $n$-th iterate of the map takes the form
\begin{equation} \label{Eq:fepsn_gf}
\left\{
\begin{aligned}
	I_n&=I_0-\varepsilon \sum_{k=0}^{n-1}\partial _2s(I_{k+1},\varphi_k),\\
	\varphi_n&=\varphi_0+\sum_{k=1}^{n}h_0'(I_k)+\varepsilon\sum_{k=0}^{n-1} \partial_1 s(I_{k+1},\varphi_k) .
\end{aligned}	
\right.
\end{equation}
The subsequent analysis is motivated by the application of the standard scaling 
near the resonant torus $I=I_*$ with the scaled action $J$ defined by
the equation $I=I_*+\sqrt{\varepsilon}J$. In the scaled variables, the map
\[(\varphi,J)\mapsto (\bar\varphi,\bar J) = (\varphi_n-n\omega_*,(I_n-I_*)/\sqrt{\varepsilon})\] takes the form
\begin{equation} \label{Eq:fepsn_scaled}
\left\{
\begin{aligned}
	\bar J&=J-\varepsilon^{1/2} \sum_{k=0}^{n-1}\partial _2s(I_{k+1},\varphi_k),\\
	\bar \varphi&=\varphi+\sum_{k=1}^{n}(h_0'(I_k)-\omega_*)+\varepsilon\sum_{k=0}^{n-1} \partial_1 s(I_{k+1},\varphi_k) ,
\end{aligned}	
\right.
\end{equation}
where $(I_k,\varphi_k)$ denote the trajectory of the point
$(\varphi_0,I_0)=(\varphi,I_*+\sqrt \varepsilon J)$ under the original map. It
is not too difficult to see that  on a bounded domain
the scaled map is $O(n\sqrt{\varepsilon})$-close to the identity. 
The interpolating vector field of order one  is explicitly represented  by the formula above as
$X_1=(\bar J-J,\bar\varphi-\varphi)$. Expanding $X_1$ into Taylor series
in powers of $\sqrt{\varepsilon}$ we see that the leading term is of the order of
$\sqrt{\varepsilon}$ and, in agreement with the general theory of Lemma~\ref{mjet1}, it is Hamiltonian
with the Hamiltonian function  
\[
\hat H_1(J,\varphi)=\sqrt{\varepsilon}n \bigl(
K(J)+V_*(\varphi)
\bigr)
\]
where
\begin{equation}
\label{EQ:KV}
K(J)=\frac12(h_0''(I_*) J)\cdot J\qquad\mbox{and}\qquad
V_*(\varphi)=\frac1n\sum_{k=0}^{n-1}s(I_*,\varphi+k\omega_*).
\end{equation}
The function $K(J)$ comes from the quadratic part of the Taylor expansion of $h_0$
around $I=I_*$ while the linear part of the expansion vanishes due to the equality $h_0'(I_*)=\omega_*$. We see that the potential $V_*$ coincides with the average of the generating function $s$ over the unperturbed periodic orbit on the resonant torus.
The strong convexity of $h_0$ provides a  lower bound for $K$ so we have that
\[
\frac\nu2|J|^2\le K(J)\le  \frac{\nu_2}{2}|J|^2= \frac d2 \| h_0''\| \,|J|^2.
\]
The function $E(J,\varphi)=\frac1{n\sqrt{\varepsilon}}\hat H_1(J,\varphi)$
defines a slow variable in a neighbourhood of the resonance, it is constant along orbits of the flow of $\hat H_1$  and it changes slowly 
under  iterates of the map \eqref{Eq:fepsn_scaled}.
If $E_0>\max V_*$, then the set $E\le E_0$ contains a ball $|J|\le \hat r_0$
provided $\tfrac12 \nu_2 \hat r_0^2\le E_0-\max V_*$. Let $E_1>E_0$. The set $E\le E_1$ is contained in the ball $|J|\le r_1$ provided $\tfrac12\nu r_1^2\ge E_1-\min V_*$. Consequently, if 
the initial point satisfies $|J_0|\le \hat r_0$ then its energy $E(J_0,\varphi_0)\le E_0$
and if some iterate satisfies $|J_k|>r_1$ then its energy $E(J_k,\varphi_k)\ge E_1$.

Since $|V_*|\le\|s\|$ we can choose 
$E_0=2\|s\|$,  $E_1=4\|s\|$, $\hat r_0^2=2\nu_2^{-1}\|s\|$ and $r_1^2=10\nu^{-1}\|s\|$. 
Then  we can conclude that a trajectory with initial condition satisfying
$|J_0|\le r_0$ remains in the ball $|J|\le r_1$ while the changes in $E$ do not exceed $2\|s\|$.

In order to achieve exponential estimates for the stability times we need to consider the optimal approximation for the map by an autonomous Hamiltonian flow instead of the leading order approximation discussed above.
For this purpose 
we can 
use Theorem~\ref{Thm:alaNeishtadt} to get an embedding into a Hamiltonian flow
with exponentially small error. We can repeat the arguments of
Section~\ref{Se:intrepol_at_Torus} replacing in the definition of the domain
$\mathcal D(I_*)$ the radius $\rho_n$ by $\hat\rho_n=R_*\sqrt{\varepsilon}$. 
We choose a constant $R_*>r_1$ and $\gamma>R_*$. Then 
 $\hat\rho_n\le \rho_n$ and consequently we already know that the scaled map is analytic and the a-priori bounds \eqref{n_it_estimates} remain valid.
Using Cauchy estimates for the derivatives of the function $s$ we get
from~\eqref{Eq:fepsn_scaled}
\[
|\bar J-J|\le \frac{2\|s\|n\sqrt\varepsilon }{r}\quad\mbox{and}\quad
|\bar\varphi-\varphi|\le 2\|h_0''\| R_* n\sqrt \varepsilon + \frac{2\|s\|n\varepsilon }{\sigma}.
\]
Taking $R_*^2=11\nu^{-1}\|s\|$ and using that $\varepsilon\le 1$ we get
\[
|\bar J-J|\le C_0 n\sqrt\varepsilon\quad\mbox{and}\quad
|\bar\varphi-\varphi|\le C_0 n\sqrt\varepsilon
\]
where $C_0$ can be easily expressed in terms of $\|s\|$, $r$, $\sigma$, $\nu$ and $\|h_0''\|$.

Then using arguments similar to the previous section we arrive to the following theorem.

\begin{thm}\label{Thm:nucleus1}
 Under the assumptions of Theorem~\ref{Thm:longstab}, 
 there are constants $c_4,c_5$ independent of the resonance
 such that if $|I_0-I_*|^2\le 2\nu_2^{-1}\|s\|\varepsilon$
 then
 \[
|I_k-I_*|^2\le 11\nu^{-1}\|s\|\varepsilon \quad\mbox{for}\quad
0\le k \le \hat T_\varepsilon=c_4\exp(c_5/\sqrt{n^2\varepsilon}).
 \]
\end{thm}

It should be noted that the proof of this theorem is  a refinement
of the proof of Theorem~\ref{Thm:longstab}. Both theorems cover the same set 
of fully resonant tori and 
for each one Theorem~\ref{Thm:nucleus1} provides a nucleus, a smaller stability zone 
with longer stability times. For a fixed $n$ the difference becomes more prominent as $\varepsilon$ decreases.
 The estimate suggests that Arnold diffusion slows
down substantially in a neighbourhood of resonances of maximal multiplicity
provided the period $n$ is not too high for a given $\varepsilon$.

In Theorem~\ref{Thm:nucleus1} the constant $c_5$ is chosen to be the same for 
all resonances. 
It should be noted that for some resonances the bounds
for the stability times can be substantially improved (note that doubling $c_5$
is equivalent to squaring a very large number $\hat T_\varepsilon$).
Indeed, at the centre of our proofs are the upper bounds
for the sums in the right-hand side of the equation \eqref{Eq:fepsn_scaled}
which are used to control the distance of the map from the identity.
These sums can be interpreted as average values of functions taken over a finite segment of a trajectory of the map and we used elementary but not always
optimal bounds. 
For example, we used
 that $|V_*|\le\|s\|$, which does not take into account 
 that the
average value of a periodic function can be much smaller. 
A sharper bound can be
obtained if we take into account properties of the frequency vector $\omega_*$.
We notice that the function
$V_*$ inherits periodicity in $\varphi$ from the function $s$ and in addition $n\omega_*\in\Z^d$ implies that for all $\varphi$
\[
V_*(\varphi+\omega_*)=V_*(\varphi).
\]
It follows easily that all non-resonant Fourier coefficients of $V_*$ must vanish,
i.e., if $j\cdot\omega_*\notin \Z$ for some $j\in\Z^d$ then the Fourier expansion of $V_*$
does not have a term proportional to $\exp(2\pi i\, j\cdot \varphi)$.
In terms of Fourier expansions we can write
\[
V_*(\varphi)=\sum_{j\cdot\omega_*\in \Z} s_j(I_*)\mathrm e^{2\pi i \,j\cdot\varphi}.
\]
Since $s$ is an analytic function of $\varphi$, its Fourier coefficients $s_j(I_*)$
decay exponentially fast when $|j|_1$ grows. Therefore the amplitude of $V_*$
can be substantially smaller than $\|s\|$.

This observation suggests that in the absence of low order resonances
the stability times should be much larger than the general lower bound $\hat T_\varepsilon$.
This situation can arise either due to the properties of the frequency vector $\omega_*$
or due to the absence of the resonant terms in the Fourier expansion of the generating function. 
In particular, we expect that in the latter case
 the lower bound for the stability time scales as $\exp(c\varepsilon^{-\alpha}/n)$ with  $\alpha>\frac12$ (a phenomenon similar to  \cite{S1994}).
On the other hand, in the case when a full spectrum condition is satisfied
the lower bound for the stability time scales as $\exp(c_5\varepsilon^{-1/2}/n)$
with a constant $c_5\sim \mathrm e^{\pi r j_0}$, i.e. the constant becomes very large for
larger values of $j_0$, the order of the lowest order resonance of $\omega_*$.

\section{Final comments and conclusions \label{Se:conclusion}}

Our proof of the Nekhoroshev estimates  is based on
discrete averaging. The weighted averages of iterates of the near-integrable map are
explicitly computed to produce the interpolating vector field $X_m$.
This vector field is not necessarily Hamiltonian but 
 Theorem~\ref{Thm:alaNeishtadt} states the existence of a Hamiltonian vector field 
$\hat{X}_m = J \nabla H_m$ very close to $X_m$.
In a neighbourhood of a fully resonant torus
the time-one maps of $X_m$ and $\hat X_m$ are both exponentially close to 
$f^{n}_\varepsilon$ for $n< N_\varepsilon$ and, consequently,
 the map preserves $H_m$
up to an exponentially small error.
The convexity arguments 
are used to show that level lines of $H_m$ present obstacles for the drift of action
variables.

The analytical tools developed in this paper rely on explicit constructions
and provide a useful tool for  analytical and numerical exploration of
long term dynamics.

{\bf \em Computing a slow variable from iterates of the map in original variables.}
The value of $H_m$ is a natural slowly moving observable which provides
a useful instrument for studying long time stability and Arnold diffusion. Our
method provides an explicit expression for this slow variable in terms of
weighted averages of the iterates of the map, hence avoiding transformations of
coordinates traditionally used to reduce the system to a normal form. 
In particular,  we may construct the interpolating vector field $X_m$
using the iterates of $f_\varepsilon^n$ in the original coordinates
$(I,\varphi)$, see Remark~\eqref{Rk_scaling}.

{\bf \em Formal embedding of a near-the-identity map into an autonomous flow.}
Our method provides a new algorithm 
for constructing the formal embedding of a
near-the-identity map into an autonomous flow. For example, 
let
$f_\mu:(p,q)\mapsto(p_1,q_1)$ be defined with the help of a generating
function
\[
G_\mu(p_1, q) = p_1 q + \mu S(p_1, q).
\]
We can get an explicit expression for $\hat X_{m,\mu}$ for any $m$
by differentiating $m$ times with respect to $\mu$ the system
\[
\left\{
\begin{aligned}
p_{k} &= p_{k-1} - \mu \frac{\partial{S}}{\partial q}(p_k,q_{k-1}),\\
q_k &= q_{k-1} + \mu \frac{\partial{S}}{\partial p}(p_k,q_{k-1})
\end{aligned}
\right.
\]
for $k=1,\dots,m$, and evaluating at $\mu=0$.
The derivatives depend in a polynomial way on
partial derivatives of $S$ and can be computed explicitly. 
Then the Hamiltonian $H_m$ can be restored from the vector field. 
For example, the second order interpolating Hamiltonian for $f_\mu$ is
\[
H_{2,\mu}= \mu S
-\mu^2 \, \frac12 \, \frac{\partial  S}{\partial p} \cdot
\frac{\partial  S}{\partial q}
\]
where all functions are evaluated at a point $(p,q)$.   
This argument can also be  applied to an individual map with a generating function
$G(p_1, q) = p_1 q +  S(p_1, q)$. 
This map is approximated by the time one map of the flow defined by
\[
H_{2}= S
-\frac12 \, \frac{\partial  S}{\partial p} \cdot
\frac{\partial  S}{\partial q}
\]
with the error cubic in $\epsilon=\|\nabla S\|$ and explicitly computable
constants according to Remark~\ref{remark_cota}.

{\bf \em Numerical evaluation of $H_m$.}
In numerical computations it is usually not convenient to rely on algebraic manipulations and instead one can evaluate $H_m(x)$ from integrals of $X_m$
along 
continuous paths connecting a base point $p$ and the point $x$, see
\cite{GelfreichV18}. Note that this procedure typically produces a Hamiltonian
which is not periodic in the angle variables even when $X_m$ is periodic.
The periodicity of the Hamiltonian
can be restored by adding a small correction.

{\bf \em The choice of the interpolation scheme.}
In this paper we have used the Newton interpolation scheme to obtain $X_m$. This
 scheme uses the forward orbit $x_0,\ldots,x_m$ for the construction of
$X_m$ and simplifies some analytical expressions involved in the proof. But any
other interpolation scheme will lead to  similar results. 

From the numerical point of view,  higher accuracy of
interpolation is expected when interpolation nodes are located symmetrically
around $x_0$.
This can be useful for numerical studies of concrete examples when
relatively large values of $\varepsilon$ are to be used in order to observe
Arnold diffusion on a time scale accessible to the computer.

For example  we can use the Gauss forward formula
\[
P_m(t;x_0)=x_0+\Delta_1(x_0)t+\frac{\Delta_2(x_{-1})}{2!}t(t-1)
+\frac{\Delta_3(x_{-1})}{3!}(t+1)t(t-1)
+\dots
\]
If $m=2j$ is even, then the interpolation is based on a symmetrical
piece of the orbit, $x_{-j},\dots,x_0,\dots,x_j$.
Differentiating with respect to $t$ at $t=0$ we obtain the interpolating vector field
\begin{equation} 
\label{Eq:interpol_VF_symmetric}
X_m(x_0)=\sum_{k=1}^{j}  (-1)^{k-1}\left(\frac{(k-1)!^2\Delta_{2k-1}(x_{-k+1})}{(2k-1)!}
-\frac{(k-1)!k!\Delta_{2k}(x_{-k})}{(2k)!} \right)
\end{equation}
instead of \eqref{Eq:interpol_VF}.
This interpolating vector field can be used
in Theorem~\ref{Thm:alaNeishtadt} provided two steps in the proof
 are modified.
First, the constant in the bound \eqref{norD1opt2} depends on the coefficients of $P_m$.
Using \eqref{Eq:interpol_VF_symmetric} and the bound \eqref{bound_D_k} for the
finite differences it is easy to check that for the Gaussian symmetric scheme
the same upper bound holds. 
Second, given a number $m$ of iterates of the original map,
the Gauss symmetric scheme allows us to double the value of $\mu_m$ in the proof of
Theorem~\ref{Thm:alaNeishtadt}. This leads to better accuracy of the embedding of the map into a
flow with the error being of the order $\sim
\exp(-\delta/3e\varepsilon)$, i.e. we get the error term approximately squared.

{\bf \em Nucleus of resonances.} 
The fact that the methodology to obtain the estimates on the long term dynamics
of the map does not depend on changes of coordinates leads to a description of
the leading order  dynamics near the nucleus of the resonances, that is, in a ball of radius $\mathcal{O}(\sqrt{\varepsilon})$ near a resonant torus $I=I_*$. 
The corresponding energy preservation leads to much larger stability times for initial conditions
in the nucleus of the resonances. The construction is explicit. In particular,
the potential part of the Hamiltonian is given by the average of the generating
function $s(I,\varphi) $ of $f_\varepsilon$ along the unperturbed periodic orbit
corresponding to the resonant torus $I=I_*$.

\medskip

Finally we note that we have used the convexity assumption for
the generating function of the unperturbed map. 
At the present time it is not clear up to which extent the convexity assumption can be relaxed in our proof.
Nevertheless we expect that our method can be useful for studying systems without
the convexity assumption. In this case our method produces 
a very slow variable near a resonance but the corresponding level
lines  are not necessarily an obstacle for the movements of actions.
Nevertheless, the slow variables may provide useful information on possible directions of Arnold diffusion. We also hope that our method can be applied to study dynamics of near integrable systems without references to action-angle variables for the integrable part, opening potential applications to study dynamics of maps in neighbourhoods of totally elliptic fixed points.

\appendix

\section{Implicit function theorem}

In the proof we switch between a symplectic map
and the corresponding generating function.
This transition relies on the following version
of the implicit function theorem.

\begin{thm}[Implicit Function Theorem]\label{Thm:IFT}
Let $A,B\subset \C^d$ be open sets, 
$f:A\times B\to\C^d$  an analytic function,
 $x_0\in A$ and $y_0\in B$. 
If there is $R>0$ such that $B_R(y_0)\subset B$ and
\[M=\sup_{y\in B_R(y_0)}|f(x_0,y)|<\frac{R}{d+1}
\]
then the equation
\[
y=y_0+f(x_0,y)
\]
has a unique solution $y\in B_R(y_0)$. Moreover, this solution 
depends analytically on~$x_0,y_0$.
\end{thm}

\begin{proof}
We use the contracting mapping theorem (the $\infty$-norm is used
for vectors in $\C^d$). 
The closed ball $\overline{B_{M}(y_0)}$ is invariant under the map \[g:y\mapsto y_0 +f(x_0,y).\] In order to check that $g$ is contracting
we take  $u,v\in \overline{B_{M}(y_0)}$,
then 
\[
\begin{split}
    |g_j(u)-g_j(v)|&=|f_j(x_0,u)-f_j(x_0,v)|\\&=\left|
    \int_0^1 \sum_{k=1}^d\partial_{y_k}f_j(x_0,u s+(1-s)v)(u_k-v_k)ds\right|
    \\&
    \le
    \int_0^1 \sum_{k=1}^d\left|\partial_{y_k}f_j(x_0,u s+(1-s)v)\right|\;|u_k-v_k|ds
    \\&\le
    \sum_{k=1}^d\frac{\|f_j\|_{B_R}}{R-M}|u_k-v_k|=\frac{ Md}{R-M}|u-v|
\end{split}
\]
where we used the Cauchy bound for the derivatives.
The inequality $M<\frac{R}{d+1}$ implies $\frac{Md}{R-M}<1$.
Therefore the map $g$ is contracting and it has a unique fixed point which depends analytically on the parameters.
\end{proof}

\section{Formal interpolation of a symplectic family \label{appendix:Hperiodic}}

Let $B$ denote a ball (or a simply connected domain) in $\R^d$
 and let $F_\mu$ be an analytic family of exact symplectic maps defined in $B\times \T^d$
 with $F_0=\mathrm{Id}$. The following theorem represents 
 a generally known statement (see e.g. \cite{BenettinG94}).
 Here we provide a more direct proof of the statement in the form needed for 
 the proof of our main theorem.

\begin{thm} \label{Thm:Hperiodic}
If $X_\mu$ is the formal  vector field
on $B\times\T^d$ such that
its formal time one map coincides with the Taylor expansion of $F_\mu$,
then there is a formal Hamiltonian $H_\mu$ with coefficients
defined on $B\times\T^d$ such that $X_\mu=\J\nabla H_\mu$.
\end{thm}

\begin{proof}
We have already proved existence of the formal vector field 
 \[
 X_\mu=\sum_{k\ge 1}\mu^k X^{(k)}(p,q)
 \]
 where the coefficients are smooth functions independent of $\mu$
 and periodic in $q$. We want to show that if $F_\mu$ is symplectic
 for every $\mu$ 
 then the formal vector field is Hamiltonian, i.e. for every $k$, $X^{(k)}=(-\partial_q h^{(k)},\partial_p h^{(k)})$
 for some  function $h^{(k)}:B\times \R^d\to\R$.
 If in addition $F_\mu$ are exact symplectic, then $h^{(k)}$
 are periodic in $q$. The proof is based on analysis of loop actions.
 
 Let $\gamma:[0,t]\to\R^{2d}$ be a smooth curve
 inside the domain of the map such that $\gamma(1)-\gamma(0)\in\{0\}\times\Z^d$, i.e. 
 $\gamma$ is a lift of a loop from $\mathbb{R}^d \times \mathbb{T}^d$. 
 Let 
$$
\mathcal{I}_F(\gamma) = \int_{F(\gamma)} p \, dq - \int_\gamma p \, dq.
$$
If $F$ is symplectic, Stokes' theorem implies that 
$\mathcal{I}_F(\gamma)$ 
depends only on the homotopy class of $\gamma$ and
$\mathcal{I}_F(\gamma)=0$ for any contractile loop $\gamma$.
 If $F$ is exact symplectic map then
$\mathcal{I}_F(\gamma)=0$ for any lift $\gamma$ of a loop from $\mathbb{R}^d \times \mathbb{T}^d$. 
For a flow defined by a vector field $X=(X_p,X_q)$ we write $\gamma_t(s)=(p(t,s),q(t,s)):=\Phi_X^t\gamma(s)$. Then
\[
\begin{split}
\frac{d}{dt} \mathcal{I}_{\Phi^t_X}( \gamma) & = 
\frac{d}{dt} \int_{\Phi_X^t(\gamma)} p \, dq \, = \,
\frac{d}{dt} \int_0^1  p(t,s) \, \frac{\partial q}{\partial s} (t,s) ds  \\
& = \int_0^1 \left(  \frac{\partial p}{\partial t}(t,s) \, \frac{\partial q}{\partial s} (t,s) + p(t,s) \, \frac{\partial^2 q}{\partial s \partial t}(t,s)  \right) ds  \\ 
& =  \int_0^1  \left(  \frac{\partial p}{\partial t}(t,s) \frac{\partial q}{\partial s} (t,s)  - \frac{\partial p}{\partial s}(t,s) \, \frac{\partial q}{\partial t} (t,s)  \right) ds 
\\
&=
\int_{\Phi_X^t\gamma(0)}^{\Phi_X^t\gamma(1)}
\left(X_pdq-X_qdp\right).
\end{split}
\]
The right hand sides vanishes for every $\gamma$ with $\gamma(0)=\gamma(1)$
iff $X$ is Hamiltonian, i.e $X_p=-\partial_q H$ and $X_q=\partial_p H$
for some function $H:B\times \R^d\to \R$. In this case the integral can be
evaluated explicitly for any lift $\gamma$,
\[
\begin{split}
\frac{d}{dt} \mathcal{I}_{\Phi^t_X}( \gamma) =
H(\Phi^t_X(\gamma(0)))
-
H(\Phi^t_X(\gamma(1)))
=
H(\gamma(0))
-
H(\gamma(1)).
\end{split}
\]
Since $ \mathcal{I}_{\Phi^0_X}( \gamma)=0$ we get
\[
{\mathcal I}_{\Phi^1_X}(\gamma)=H(\gamma(0))
-
H(\gamma(1)).
\]
Recall that the coefficients of the formal vector field are defined form the following requirement: for every $m\in\N$, a partial sum of the formal series 
\[
X_{m,\mu}(p,q)=\sum_{k=1}^m \mu^kX^{(k)}(p,q)
\]
defines the flow $\Phi^1_{X_{m,\mu}}=F_\mu + O(\mu^{m+1})$.
If $F_\mu$ are symplectic, then $\mathcal{I}_{F_\mu }( \gamma)=0$
for all contractible loops $\gamma$ and, consequently, for $k\le m$
\[
\oint_{\gamma}
\left(X^{(k)}_pdq-X^{(k)}_qdp\right)=0.
\]
Therefore $X_\mu$ is Hamiltonian.

If $F_\mu$ are exact symplectic, then ${\mathcal I}_{\Phi^1_{X_{m,\mu}}}=
\mathcal{I}_{F_\mu }( \gamma)+O(\mu^{m+1})=O(\mu^{m+1})$
for all lifts  $\gamma$. Since the Hamiltonian of $X_{m,\mu}$
is polynomial of degree $m$ in $\mu$ we conclude that all coefficients
$h^{(k)}$ are periodic in $q$.
\end{proof}

\section{Generating functions of exact symplectic maps \label{Se:genfun}}

In this paper we need to find a generating function
for a near-the-identity exact symplectic map defined on a subset of $\mathbb R^d\times\mathbb T^d$.

We recall that a map is called symplectic if it preserves
the standard symplectic form 
\[
\omega=\sum_{l=1}^d dp_l\wedge dq_l.
\]
The map is exact if it preserves the loop action
\[
A(\gamma)=\oint_\gamma \sum_{l=1}^d p_l\, dq_l
\]
for all loops inside its domain. A symplectic map
automatically preserves loop actions for contractible loops.

Suppose that $f:(p,q)\mapsto (\bar p,\bar q)$ 
is a lift of a symplectic
map. We assume that the map is analytic in a 
neighbourhood of $B_\rho\times\mathbb R^d$.
Suppose that 
we can rewrite $f$ in the cross form
\begin{equation}\label{Eq:crossform}
\begin{aligned}
 p  & =\bar  p + u(\bar p,q),\\
 \bar q&=q+v(\bar p,q),
\end{aligned}
\end{equation}
where the functions $u=(u_1,\ldots,u_d)$ and $v=(v_1,\ldots,v_d)$ are periodic in $q$.
Since the map is symplectic we get that $\sum_{l=1}^dd\bar p_l\wedge d\bar q_l=\sum_{l=1}^dd p_l\wedge d q_l$ and we get
\[
\begin{split}
 \sum_{l=1}^d d(u_l dq_l + v_ld\bar p_l)&= 
  \sum_{l=1}^d (du_l\wedge dq_l + dv_l\wedge d \bar p_l)
 \\& =\sum_{l=1}^d ((d p_l-d\bar p_l)\wedge dq_l +
 (d\bar q_l-dq_l)\wedge d\bar p_l)\\
 &=\sum_{l=1}^d d p_l\wedge d q_l -
 \sum_{l=1}^d \bar d \bar p_l \wedge dq_l=0.
\end{split}
\]
Then we choose a base point $(\bar p_0,q_0)$ and define
\begin{equation}\label{Eq:genfunc}
s(\bar p,q)=\int_{(\bar p_0,q_0)}^{(\bar p,q)}\sum_{l=1}^d (u_l dq_l + v_ld\bar p_l).
\end{equation}
The previous argument implies that for a symplectic $f$ the value of the integral is independent of the path
connecting the end points as the domain is simply connected.
Differentiating the integral we see that
\[
u_l=\frac{\partial s}{\partial q_l},\qquad v_l=\frac{\partial s}{\partial \bar p_l}.
\]
Consequently  the map $f$ can be defined with the help of the generating function
$\bar p\cdot q+s(\bar p,q)$. 
Let $e_l$ denote a vector of the canonical basis in $\R^d$.
Then
\[
s(\bar p,q+e_l)-s(\bar p,q)=
\int_{(\bar p_0,q_0)}^{(\bar p_0,q_0+e_l)}\sum_{l=1}^d (u_l dq_l + v_ld\bar p_l).
\]
Now suppose that $f$ is homotopic to the identity and
consider a smooth curve  $\gamma_l=(p(t), q(t))$ 
such that $q(1)=q(0)+e_l$. Let $\bar \gamma_l=(\bar p(t), \bar q(t))$
be the image of this curve. 
Since the map is homotopic to the identity we have $\bar q(1)=\bar q(0)+e_l$.
We compute the difference of the loop actions:
\[
\begin{split}
  A&(\bar\gamma_l)-A(\gamma _l) =
  \int_0^1 \sum_{l=1}^d \left( \bar p_l(t)\, d\bar q_l(t)
 -  p_l(t)\, dq_l(t)\right)
  \\
  &= \int_0^1 \sum_{l=1}^d \left( 
  \bar p_l(t)\, dv_l(\bar p(t),q(t))
 -u_l(\bar p(t),q(t))\, dq_l(t)\right)
 \\
  &= \int_0^1 \sum_{l=1}^d \bigl( 
- v_l(\bar p(t),q(t)) d\bar p_l(t)
 -u_l(\bar p(t),q(t))\, dq_l(t)\bigr)
  \\
  &= -\int_{(\bar p(0),q(0))}^{(\bar p(0),q(0)+e_l)} \sum_{l=1}^d \bigl( 
v_l d\bar p_l
 +u_l dq_l\bigr)=-s(\bar p(0),q(0)+e_l)+s(\bar p(0),q(0)).
\end{split}
\]
We see that the conservation of loop actions is equivalent
to the periodicity of $s$.

\section*{Acknowledgements}
A.V. is supported by the Spanish grant PID2021-125535NB-I00 funded by 
MICIU/AEI/10.13039/501100011033 and by ERDF/EU.
He also acknowledges
the Catalan grant 2021-SGR-01072 and the
Severo Ochoa and Mar\'{\i}a de Maeztu Program for Centers and Units of
Excellence in R\&D (CEX2020-0010 84-M).

\addcontentsline{toc}{section}{References}
\bibliographystyle{plain}
\bibliography{NearInteg_Nekhoroshev}

\end{document}